\documentclass[12pt]{amsart}
\usepackage{geometry}                
\usepackage[parfill]{parskip}    
\usepackage{graphicx}
\usepackage{xcolor}
\usepackage{amssymb}
\usepackage{epstopdf}

\usepackage{hyperref}
\usepackage{url}

\newcommand{\C}{\mathbb{C}}
\newcommand{\R}{\mathbb{R}}
\newcommand{\N}{\mathbb{N}}
\newcommand{\tleq}{\trianglelefteq}

\newtheorem{theorem}{Theorem}

\newtheorem{corollary}{Corollary}
\newtheorem{definition}{Definition}
\newtheorem{remark}{Remark}
\newtheorem{proposition}{Proposition}
\newtheorem{lemma}{Lemma}

\title{Majorant series  for the  $N$-body problem}

\author[M. Anto\~nana]{Mikel Anto\~nana}
\address{\hspace*{-12pt}M. Anto\~nana: University of the Basque Country (UPV/EHU),  Donostia-San Sebasti\'an, Spain.}
\email{Mikel.Antonana@ehu.eus}

\author[Ph. Chartier]{Philippe Chartier}
\address{\hspace*{-12pt}Ph.~Chartier: Univ Rennes, INRIA-MINGuS, CNRS, IRMAR-UMR 6625, F-35000 Rennes, France}
\email{Philippe.Chartier@inria.fr}

\author[A. Murua]{Ander Murua}
\address{\hspace*{-12pt}A. Murua: University of the Basque Country (UPV/EHU),  Donostia-San Sebasti\'an, Spain.}
\email{Ander.Murua@ehu.eus}

\begin{document}
\maketitle
\begin{abstract}
As a follow-up of  a previous work of the authors, this work considers {\em uniform global time-renormalization  functions} for the {\em gravitational} $N$-body problem. It improves on the estimates of the radii of convergence obtained therein by using a completely different technique, both for  the solution to the original equations and for the solution of the renormalized ones. The aforementioned technique which the new estimates are built upon is known as {\em majorants} and allows for an easy application of simple operations on power series.  The new radii of convergence so-obtained are approximately doubled with respect to our previous estimates. In addition, we show that {\em majorants} may also be constructed to estimate the local error of the {\em implicit midpoint rule} (and similarly for  Runge-Kutta methods) when applied to the time-renormalized $N$-body equations and illustrate the interest of our results for numerical simulations of the Solar System.
\end{abstract}
\section{Introduction}
We are concerned here with the solution of the $N$-body problem: considering $N$ masses $m_i, \, i = 1,\ldots,N$, moving in a three-dimensional space under the influence of gravitational forces, Newton's law describes the evolution of their {\em positions} $q_i$ and  {\em velocities} $v_ i$ for $i = 1,\ldots,N$ through the equations
\begin{equation} 
\label{eq:Nbody2}
m_i \frac{d^2 q_i}{dt^2}= \sum_{j \neq i} \frac{G m_i m_j}{\|q_j-q_i\|^3} (q_j-q_i),
\end{equation}
where $G$ is the gravitational constant and $\|q_j-q_i\|$ is the distance between $q_i$ and $q_j$ in the Euclidean norm of $\R^3$.

The $N$-body equations are of great importance in physics and celestial mechanics in particular and the question of representing their solution  in the form of a convergent series has been a long-standing problem. For an introduction on the historical and practical aspects of the subject, we refer to \cite{antonana2020}.  To make a long story short, let us just recall that the problem was solved for $N = 3$ by Karl Frithiof Sundman \cite{sundman} and Qiu-Dong Wang for the $N \geq 3$-case~\cite{wang} in the 1990s.

As an intermediate step to obtain a representation of each solution of (\ref{eq:Nbody2}) in terms of a series expansion convergent for all $t$ in its maximal interval of existence, both authors rewrite the $N$-body equations (\ref{eq:Nbody2}) in terms of
 a new independent variable $\tau$ related to the physical time $t$ by 
\begin{equation} \label{eq:t-r}
\frac{d\tau}{d t} = s(q(t))^{-1}, \quad \tau(0)=0,
\end{equation}
for an appropriate {\em time-renormalization function} $s(q)$ depending on the positions 
$$
q = (q_1,\ldots,q_N).
$$
Note that time-renormalization is also referred to as   {\em time-transformation} and has been studied by several authors in the context of geometric integration \cite{hairer97,huang97,mikkola97,calvo98,hairer05}. 
 In \cite{antonana2020}, we proposed new time-renormalization functions for the purpose of simulating the  $N$-body problem with constant stepsizes without degrading the accuracy of the computed trajectories. In contrast with previously known functions \cite{sundman, wang}, ours  depend not only on positions but also on velocities. 
 
 Our time-renormalized equations were derived  by considering estimates of the domain of existence of the holomorphic extension of maximal solutions of  the $N$-body problem (\ref{eq:Nbody2}) to the complex domain. Noticeably, these global time-renormalizations were shown to be {\em uniform}  in the sense that the solution of the time-renormalized equations in the fictitious time $\tau$ can be extended {\em analytically} to the strip 
$$
\{\tau \in \mathbb{C}\ : \ |\mathrm{Im}(\tau)| \leq \beta\}
$$
for some $\beta>0$ independent of the initial conditions $(q^0,v^0) \in \mathbb{R}^{6N}$ (provided that $q^0_i \neq q^0_j$ for all $i \neq j$) and the masses $m_i$, $1\leq i \leq N$.

Our main contribution in this paper is to explicitly construct {\em majorants} for the expansions of the solution of the $N$-body problem as a series in powers of either $t$ or $\tau$, in addition to majorants for the expansion of discrete solutions of the time-renormalized equations in series of powers of the step-size. That technique, allows us to prove the our global time-renormalized equations are uniform (in the sense above) with a larger value of $\beta$. Furthermore, we obtain a similar result for the discrete solution, which provides a solid theoretical basis for the application of numerical schemes with constant time-steps to the time-renormalized equations.

Although not fully resorting to {\em Geometric Numerical Integration}, our work is related to this special issue as it paves the way for the development of geometric integrators for the $N$-body problem. As a matter of fact, constant step-size integration in physical time is not feasible for trajectories with close approaches (such as those encountered in gravitational problems), as it would require unaffordable computations. In contrast,  uniform global time-renormalization allows to numerically integrate efficiently the $N$-body problem with fixed step-sizes, as usually required for the geometric numerical integration of {\em reversible systems} and Hamiltonian systems. The results on the discretization of time-renormalizations considered in the present work are directly applicable for schemes that preserve the reversible character of the $N$-body problem. However, the time-renormalized equations we consider in this work do not keep the Hamiltonian character of the $N$-body equations in physical time, and thus are not appropriate for symplectic schemes. The standard way to proceed (see for instance  \cite{hairer06, leimkuhler05,mikkola97}) for symplectic integrators is to consider a Hamiltonian system in the extended phase space obtained by considering two additional state variables: the physical time $t$ and its conjugate momentum $\mathcal P$. The time-renormalization transformation (\ref{eq:t-r}) from the physical time $t$ to the fictitious time $\tau$ is achieved by considering the Hamiltonian
$$
\Gamma(q,p,t,\mathcal P) =s(q,v) (H(q,p)+\mathcal P),
$$
and choosing the initial value for $\mathcal P$ as $\mathcal P(0)=-H(q(0),p(0))$. We believe that the  time-renormalization functions proposed in the present work are suitable also for the application of symplectic schemes with constant step-size to the Hamiltonian system in the extended phase space. However, the construction of majorants for the expansion of the discrete solutions of the Hamiltonian time-renormalized equations in series of powers of the step-size requires a specific treatment that is outside the scope of the present work.

Generally speaking, a power series 
$f = \sum_{k\geq 0} f_k\, t^k$ with coefficients $f_k$ in $\R^n$ is said to be {\em majored} by  $\bar f = \sum_{k\geq 0} \bar f_k\, t^k$ with coefficients $\bar f_k$  in $\R_+$, 
if, for all $k \in \N$, 
$$
\|f_k\| \leq \bar f_k
$$
where $\|\cdot\|$ is the Euclidean norm in $\R^n$ and we then write 
$$
f \unlhd \bar f. 
$$
The application of the technique of majorant equations goes back to the proof of  the Cauchy-Kovalevskaya theorem~\cite{Car61, Pet50, vK75} (see also \cite{JvdH} for a specific application to ordinary differential equations) and has several advantages in our context:
\begin{itemize}
\item it allows for an easy estimate of the radius of convergence of the series $f$;
\item simple rules on majorant series apply to  the usual operations on power series, such as addition, multiplication, derivation and integration;
\item when used for the time-renormalized $N$-body equations, it leads to improved estimates of the value of $\beta$;
\item it can be used to analyze numerical discretizations of the time-renormalized equations, and in particular, to obtain bounds for their local errors.
\end{itemize}

We conclude this introductory section with the outline of the article. 
In Section \ref{sect:bpom} we will define and derive the main rules that apply to majorants.  We then construct majorants for the power expansions of the solutions of (\ref{eq:Nbody2}) in Section~\ref{sect:Nbpeipt}, and accordingly, majorants for the power expansions of the solutions of renormalized equations in Section \ref{sect:Nbptre}. In Section \ref{sect:DottrNbe} we give a majorant series for the implicit mid-point rule discretization (and more generally for arbitrary Runge-Kutta discretizations,) which leads to uniform bounds for the local errors. In Section \ref{sect:ne} we illustrate our results for the numerical simulation of a $15$-body  model of the Solar System.

\section{Basic properties of majorants} \label{sect:bpom}
We denote by  $\R^n[[t]]$  the set of {\em formal} power series in $t$ with coefficients in $\R^n$. 
Given $f = \sum_{k\geq 0} f_k\, t^k \in \R^n[[t]]$,  
we denote  
$$
f' = \sum_{k\geq 0} (k+1)\, f_{k+1}\, t^{k} \in \R^n[[t]],
$$
and
$$
\int f = \int_0^t \left( \sum_{k\geq 0} f_k\, s^k \right)\, ds =  \sum_{k\geq 1} \frac{1}{k} f_{k-1}\, t^k \in \R^n[[t]].
$$

\begin{definition} \label{defunlhd}
Given $f = \sum_{k\geq 0} f_k\, t^k \in \R^n[[t]]$  and $\bar f = \sum_{k\geq 0} \bar f_k\, t^k \in \R[[t]]$, we say that $f$ is {\em majored} by $\bar f$ and we write 
$$
f \unlhd \bar f 
$$
if, for all $k \in \N$, 
$$
\|f_k\| \leq \bar f_k
$$
where $\|\cdot\|$ is the Euclidean norm in $\R^n$. 
\end{definition}
\begin{remark}
  If $f \unlhd \bar f$, then the coefficients of the majorant series $\bar f$ are necessarily non-negative, that is,
   $\bar f \in \R_+[[t]]$, where $\R_+ = \{x \in \R\ : \ x\geq 0\}$.
\end{remark}

\begin{proposition} \label{prop:majo}
Let $f, g \in \R^n[[t]]$,  $h \in \R[[t]]$ and $\bar f, \bar g, \bar h \in \R_+[[t]]$. Then 
\begin{align}
& f \unlhd \bar f \mbox{ and } \bar f \unlhd \bar g \Longrightarrow f \unlhd  \bar g,
\label{eq:trans} \\
& f \unlhd \bar f \; \mbox{ and }  \; g \unlhd \bar g \Longrightarrow f+g \unlhd \bar f + \bar g, \label{eq:add} \\
& f \unlhd \bar f  \; \mbox{ and }  \; h \unlhd \bar h \Longrightarrow  f \; h  \unlhd \bar f \; \bar h, 
\label{eq:prod} \\
& f \unlhd \bar f  \; \mbox{ and }  \; g \unlhd \bar g \Longrightarrow \langle f,  g \rangle   \unlhd \bar f \; \bar g, \label{eq:scal} \\
& f \unlhd \bar f  \Longrightarrow \|f\|^2 \unlhd \bar f^2, \label{eq:norm} \\
& f \unlhd \bar f  \Longrightarrow f' \unlhd \bar f', \label{eq:deriv} \\
& f \unlhd \bar f  \Longrightarrow \int f \unlhd \int \bar f. \label{eq:int}
\end{align}
\end{proposition}
\begin{proof}
The assertions (\ref{eq:trans}) and (\ref{eq:add}) are immediate consequences of Definition \ref{defunlhd}. 
The assertion (\ref{eq:prod}) follows from the Cauchy product of series: if $f=\sum_{l \geq 0} t^l f_l$ and $h=\sum_{l \geq 0} t^l h_l$, then 
$$
h f  = \sum_{l \geq 0} t^l  \sum_{k=0}^l h_k f_{l-k}
$$
so that 
$$
\|(h f)_l\| \leq   \sum_{k=0}^l \|h_k f_{l-k} \| \leq  \sum_{k=0}^l  |h_k| \|f_{l-k} \| \leq  \sum_{k=0}^l  \bar h_k \bar f_{l-k}  = (\bar h \bar f)_l.
$$
As for (\ref{eq:scal}), we  write 
$$
f = \sum_{k \geq 0} t^k f_k \quad \mbox{ and } \quad g = \sum_{l \geq 0} t^l g_l
$$
so that 
$$
\langle f, g\rangle = \sum_{l \geq 0} t^l  \sum_{k=0}^l  \langle f_k, g_{l-k} \rangle.
$$
Now, by Cauchy-Schwartz inequality, we have 
$$
|\langle f_k, g_{l-k} \rangle| \leq \|f_k\| \|g_{l-k} \| 
$$
so that  
$$
|(\langle f, g\rangle)_l | \leq   \sum_{k=0}^l |\langle f_k, g_{l-k} \rangle | \leq  \sum_{k=0}^l  \|f_k\| \|g_{l-k} \|  \leq  \sum_{k=0}^l  \bar f_k \bar g_{l-k}  = (\bar f \bar g)_l.
$$
Obviously, (\ref{eq:norm}) follows from (\ref{eq:scal}).

As for (\ref{eq:deriv}) and (\ref{eq:int}), we have 
$$
f'= \sum_{k \geq 0^*} t^{k-1} k f_k \quad \mbox{ and } \quad \int f = \sum_{k \geq 0} t^{k+1} \frac{1}{k+1} f_k
$$
so that 
$$
\forall k \geq 0, \quad |f'_k| = (k+1) |f_{k+1}| \leq (k+1) |\bar f_{k+1}| = |(\bar f)'_k|
$$
and 
$$
\forall k \geq 1, \quad  \left(\int f\right)_k = \frac{1}{k} |f_{k-1}| \leq  \frac{1}{k} |\bar f_{k-1}| =  \left(\int \bar f\right)_k.
$$
\end{proof}

\begin{proposition}
If $f=1 + \sum_{l \geq 1} t^l f_l \in \R[[t]]$, then, for any $\nu \in \R$, $p=(f)^\nu = 1 + \sum_{k \geq 1} t^k p_k \in \R[[t]]$, where
\begin{align} \label{eq:p}
\forall k\geq 1, \quad p_k = \frac{1}{k} \sum_{j=0}^{k-1} ((k-j)\, \nu - j) f_{k-j} \, p_j.
\end{align}
Morevover, for any $\nu <0$, we have 
\begin{align} \label{eq:ffb}
f \unlhd \bar f \Longrightarrow f^\nu \unlhd  (2-\bar f)^\nu.
\end{align}
\end{proposition}
\begin{proof}
We first observe that $p_0  = 1$ and the equation
$$
p'(t) f(t) = \nu f^\nu(t) f'(t),
$$
obtained by differentiation of $p=f^\nu$ and multiplication by $f$,
implies (\ref{eq:p}). This very same formula for $\bar p = (2-\bar f)^\nu$ now gives 
$$
\bar p_k = \frac{1}{k} \sum_{j=0}^{k-1} ((k-j)\, \nu - j) (-\bar f_{k-j}) \bar p_j = \frac{1}{k} \sum_{j=0}^{k-1} (j-(k-j)\, \nu) \bar f_{k-j} \bar p_j 
$$
and for $\nu  < 0$, it is clear that $\bar p_k \geq 0$ and furthermore, under the assumption $f \unlhd \bar f$, that
$$
|p_k| \leq \frac{1}{k} \sum_{j=0}^{k-1} |(k-j)\, \nu - j| | f_{k-j}| |p_j|  \leq  \frac{1}{k} \sum_{j=0}^{k-1} |(k-j)\, \nu - j| \; \bar f_{k-j} \; |p_j|
$$
which, by  an induction argument, implies relation (\ref{eq:ffb}).
\end{proof}


\section{$N$-body problem: Equations in physical time} \label{sect:Nbpeipt}

Let us consider the Newtonian $N$-body gravitational problem, 
\begin{equation}
\label{eq:Nbody}
\begin{split}
     \frac{d q_i}{d t} &= v_i, \\
  \frac{d v_i}{d t} &= g_i(q_1,\ldots,q_N),
\end{split}
\end{equation}
with
\begin{equation}
\label{eq:gi}
g_i(q_1,\ldots,q_N)=\sum_{\substack{j=1 \\ j\neq i}}^N \frac{G\, m_j}{||q_{i} - q_{j}||^3}(q_j-q_i)
\end{equation}
for $i=1,\ldots,N$ and where each $q_i\in \R^3$ represents the coordinates of the $i$-th body and $v_i \in \R^3$ its velocity.

Since the right-hand side of (\ref{eq:Nbody}) is smooth provided that $q_i \neq q_j$ for all  $1 \leq i < j \leq N$,  the equations (\ref{eq:Nbody}) supplemented with the initial conditions
\begin{equation}
\label{eq:icond}
q_i(0) = q_i^0, \quad v_i(0) = v_i^0, \quad i=1,\ldots,N,
\end{equation}
admit a unique formal solution as a series in powers of $t$ (that is, $q_i \in \R^3[[t]]$, $v_i \in \R^3[[t]]$, $i=1,\ldots,N$) for regular initial values, that is, provided that 
\begin{equation}
\label{eq:icondcond}
q_i^0 \neq q_j^0\quad \mbox{for all} \quad 1 \leq i < j \leq N.
\end{equation}

We denote $q = (q_1,\ldots,q_N)$ and $v = (v_1,\ldots,v_N)$. For later use, we also denote
\begin{equation}
\label{eq:munu}
\mu(q,v) = \max_{1 \leq i < j \leq N} \frac{\|v_i- v_j\|}{ \|q_i - q_j\|}, \quad
\nu(q) =  \max_{1 \leq i < j \leq N}  \frac{M_{ij}(q)}{ \|q_i- q_j\|},
\end{equation}
where for $1 \leq i < j \leq N$,
\begin{equation}
\label{eq:Ki}
K_i(q) = \sum_{\substack{j=1 \\ j\neq i}}^N  \frac{G\, m_j}{\|q_i - q_j\|^2},  \quad M_{ij}(q) = K_i(q)+K_{j}(q).
\end{equation}

\begin{lemma} \label{lem:phys}
Consider the power series expansion 
$$
(q,v) = (q_1,\ldots,q_N,v_1,\ldots,v_N) \in \R^{6N}[[t]]
$$
of the solution of (\ref{eq:Nbody})--(\ref{eq:icond}) with (\ref{eq:icondcond}).
If the power series
\begin{equation} 
\label{eq:rho1}
\rho = 1 + \sum_{k\geq 1} \rho_k\, t^k  \in  \R_+[[t]]
\end{equation}
is a majorant of $\frac{q_i-q_j}{\|q_i^0 - q_j^0\|}$ for all $i \neq j$, i.e.
\begin{equation}
\label{eq:qijtleq}
  \forall 1 \leq i < j \leq N, \quad q_i - q_j \tleq \|q_i^0 - q_j^0\|\, \rho,
\end{equation}
then,
\begin{equation*}
 \forall 1 \leq i \leq N, \quad g_i(q) \tleq
 K_i(q^0)
  \frac{\rho}{(2-\rho^2)^{3/2}}.
\end{equation*}
\end{lemma}
\begin{proof}
Using (\ref{eq:qijtleq}) in combination with (\ref{eq:norm})  and $\bar f = \|q_i^0-q_j^0\| \rho$ implies 
$$
\forall (i,j) \in \{1, \ldots,N\}^2, \quad \|q_i-q_j\|^2  \unlhd \|q_i^0 - q_j^0\|^2 \rho^2,
$$
i.e. 
$$
\forall i \neq j, \quad \frac{\|q_i-q_j\|^2 }{\|q_i^0-q_j^0\|^2}   \unlhd \rho^2.
$$
Upon using (\ref{eq:ffb}) with $\nu=-3/2$ we obtain 
$$
\forall i \neq j, \quad \frac{\|q_i^0-q_j^0\|^3}{\|q_i-q_j\|^{3}}   \unlhd (2-\rho^2)^{-3/2},
$$
which, in combination with  (\ref{eq:qijtleq}) and using (\ref{eq:prod}), leads to 
$$
\forall i \neq j, \quad \frac{q_j-q_i}{\|q_i-q_j\|^{3}}    \unlhd \frac{1}{\|q_j^0-q_i^0\|^2} \frac{\rho}{(2-\rho^2)^{3/2}}
$$
and finally to the result by multiplying both sides by $h=\bar h = G m_j$ (using again (\ref{eq:prod})) and then summing over all $j \neq i$ (using  (\ref{eq:add})) . 
\end{proof}

\begin{lemma}
\label{lem:Phi}
Let $\mathcal{A}$ be the set of power series of the form (\ref{eq:rho1}) satisfying (\ref{eq:qijtleq}), and consider the operator
\begin{equation}
\label{eq:Phi}
\begin{array}{rcl}
\Phi: \mathcal{A} &\rightarrow&  \R[[t]] \\
\rho &\mapsto& 1 + \mu(q^0,v^0)\,  t + \nu(q^0) \, \int \int   \frac{\rho}{(2-\rho^2)^{3/2}},
\end{array}
\end{equation}
where $\int \int$ denotes applying the operator $\int$  twice.
Then $\Phi(\mathcal{A}) \subset \mathcal{A}$.
\end{lemma}
\begin{proof}
Owing to Lemma \ref{lem:phys}, we then have 
$$
 \forall 1 \leq i \leq N, \quad g_i(q) \tleq
 K_i(q^0)
  \frac{\rho}{(2-\rho^2)^{3/2}},
$$
so that, by virtue of inequality (\ref{eq:int})
$$
 \forall 1 \leq i \leq N, \quad  \int (g_i(q)-g_j(q)) \unlhd
M_{ij}(q^0) 
  \int \frac{\rho}{(2-\rho^2)^{3/2}}.
$$
and 
$$
 \forall 1 \leq i \leq N, \quad  \int \int (g_i(q)-g_j(q)) \unlhd
M_{ij}(q^0) 
 \int  \int \frac{\rho}{(2-\rho^2)^{3/2}}.
$$
Now, from the integral form of the second equation of  (\ref{eq:Nbody})  we have 
$$
v_i-v_j = v_i^0 - v_j^0 + \int \left(g_i(q)-g_j(q) \right)
$$
which translates into 
$$
v_i - v_j \unlhd  \|v_i^0 - v_j^0\| + M_{ij}(q^0)
  \int  \frac{\rho}{(2-\rho^2)^{3/2}}.
$$
Similarly, the integral form of the first equation of  (\ref{eq:Nbody})  leads to 
$$
q_i-q_j = q_i^0-q_j^0 + t\, (v_i^0-v_j^0) + \int \int \left(g_i(q)-g_j(q) \right),
$$
which in turn implies that, for $1 \leq i < j \leq N$,
\begin{align*}
  q_i - q_j &\tleq \|q_i^0 - q_j^0\| + \|v_i^0 - v_j^0\| t +  M_{ij}(q^0) \, \int \int   \frac{\rho}{(2-\rho^2)^{3/2}} \\
  &\tleq \|q_i^0 - q_j^0\| \left(1 + \mu(q^0,v^0)\,  t + \nu(q^0) \, \int \int   \frac{\rho}{(2-\rho^2)^{3/2}}\right),
\end{align*}
that is, $\rho \in \mathcal{A}$.
\end{proof}
\begin{theorem} \label{th:phys} 
The power series expansion
$$
(q,v) = (q_1,\ldots,q_N,v_1,\ldots,v_N) \in \R^{6N}[[t]]
$$
of the solution  (\ref{eq:Nbody})--(\ref{eq:icond}) with (\ref{eq:icondcond}), satisfies (\ref{eq:qijtleq}), where  $\rho  \in  \R_+[[t]]$ is the unique  power series solution of the following initial value problem
\begin{equation} 
\label{eq:rho2}
\rho'' = \nu(q^0) \,  \frac{\rho}{(2-\rho^2)^{3/2}}, \quad
\rho_0 = 1, \quad
\rho'_0 =  \mu(q^0,v^0).
\end{equation}
Furthermore, for each $i=1,2,\ldots,N$,
\begin{equation}
\label{eq:majorant_qi}
q_i - q_i^0 - t\, v_i^0  \unlhd \min_{1\leq j \leq N} \|q_i^0 - q_j^0\|\, \sum_{k=2}^{\infty} \rho_k\, t^k.
\end{equation}
\end{theorem}
\begin{proof}
Consider
$$
\rho^{[0]} = 1 + \sum_{k\geq 1} \rho^{[0]} _k\, t^k  \in  \R_+[[t]]
$$ 
such that, for all $k \geq 1$, 
$$
\rho^{[0]} _k = \max_{i \neq j} \frac{\|(q_i-q_j)_k\|}{\|q_i^0-q_j^0\|}.
$$
Here, $(q_i-q_j)_k \in \R^3$ denotes the coefficient for $t^k$ of $(q_i-q_j) \in \R^3[[t]]$, that is to say 
$$
\frac{1}{k!} \left. \frac{d^k}{dt^k} (q_i(t) -q_j(t)) \right|_{t=0}.
$$

Lemma~\ref{lem:Phi} implies that
\begin{align} \label{eq:majiter}
  \forall m \geq 1, \quad \forall 1\leq i<j\leq N, \quad q_i - q_j &\tleq \|q_i^0 - q_j^0\| \, \rho^{[m]} 
\end{align}
where $\rho^{[m]}  = \Phi(\rho^{[m-1]} )$ for $m\geq 1$.
Clearly, the operator $\Phi$ satisfies the following property~\footnote{Such an operator is called a {\em Noetherian} operator in~\cite{JvdH}.}:  for each $k\geq 1$, $(\Phi(\rho))_k$ is  a  polynomial  of the coefficients $\rho_l$ for $l \leq k-2$.

The sequence $\{\rho^{[m]} \}_{m \in \N}$ converges towards a limit
$$
\rho^{[\infty]}  = 1 + \sum_{k\geq 1} \rho^{[\infty]}_k\, t^k  \in  \R_+[[t]]
$$ 
 in the sense that for each index $k \geq 0$, the sequence $\{\rho^{[m]}_k\}_{m\in \N}$ is ultimately constant, i.e., there exists $m_k\geq 1$ such that 
$\rho^{[m]}_k=\rho^{[\infty]} _k$ for all $m\geq m_k$. Indeed,  assume that this is not the case. Let $k$ be the smallest index $l \geq 2$ such that the sequence $\{\rho^{[m]}_l\}_{m\in \N}$  is not ultimately constant.  Since  $\rho^{[m]}_k=(\Phi(\rho^{[m-1]}))_k$ is  a  polynomial  of the coefficients $\rho^{[m-1]}_l$ for $l \leq k-2$ and the sequences $\{\rho^{[m]}_l\}_{m\in \N}$ (for each for $l \leq k-2$) are ultimately constant, we get a contradiction.
 This limit is the unique solution of the fixed point equation 
$$
\rho^{[\infty]}= \Phi(\rho^{[\infty]}) 
$$
and it is thus clear from (\ref{eq:majiter}) that estimate (\ref{eq:qijtleq}) holds for $\rho=\rho^{[\infty]}$, the solution of
\begin{equation}
\label{eq:rho}
\rho = 1+ \mu(q^0,v^0)\,  t + \nu(q^0) \, \int \int   \frac{\rho}{(2-\rho^2)^{3/2}},
\end{equation}
or in other words, the unique power series solution
of (\ref{eq:rho2}). 

Finally, (\ref{eq:majorant_qi}) follows from applying Lemma~\ref{lem:phys} to
$$
q_i = q_i^0+ t\, v_i^0 + \int \int g_i(q)
$$
and taking into account (\ref{eq:rho}), which leads to
\begin{equation*}
q_i - q_i^0 - t\, v_i^0 \unlhd \frac{K_i(q^0)}{\nu(q^0)}\, \sum_{k=2}^{\infty} \rho_k\, t^k
 \unlhd \min_{1\leq j \leq N} \|q_i^0 - q_j^0\|\, \sum_{k=2}^{\infty} \rho_k\, t^k.
\end{equation*}
\end{proof}

Clearly, the solution $\rho(t)$ of (\ref{eq:rho2}) is $\rho(t) = 1 + \lambda(t\, \sqrt{\mu_0^2 + \nu_0})$ where $\lambda(t)$ is the solution of 
\begin{equation}
\label{eq:lambdaode2}
\lambda'' =  (1-\eta_0) \frac{1 + \lambda}{(1 - 2 \lambda - \lambda^2)^{3/2}}, \quad \lambda(0)=0, \quad \lambda'(0)=\sqrt{\eta_0},
\end{equation}
with
\begin{equation*}
\eta_0 = \frac{\mu_0^2}{\mu_0^2 + \nu_0},\quad
\mu_0 = \max_{1 \leq i < j \leq N} \frac{\|v^0_i- v^0_j\|}{ \|q^0_i - q^0_j\|}, \quad 
\nu_0  =  \max_{1 \leq i < j \leq N}  \frac{M_{ij}(q^0)}{ \|q^0_i- q^0_j\|}.
\end{equation*}
Since $\mu_0\geq 0$ and $\nu_0>0$, we have that $0 \leq \eta_0 <1$.  The right-hand side of that second order differential equation being analytic at $\lambda=0$,  the power series expansion of the solution
 $\lambda(t)$ of (\ref{eq:lambdaode2}) is  convergent  for small enough $t \in \R$.
  
Next, we give a formula for the radius  of convergence $r(\eta_0)$ of the power series expansion of the solution $\lambda(t)$ of (\ref{eq:lambdaode2}) for $\eta_0 \in (0,1)$. 

\begin{proposition}\label{prop:r}
For $\eta_0 \in (0,1)$, let $f$ be the function 
\begin{equation}
\label{eq:f(rho)}
f(\lambda) = \left( \eta_0 +2 \, (1-\eta_0) \, \left((1-2 \lambda-\lambda^2)^{-1/2}-1\right) \right)^{-1/2}.
\end{equation}
The differential equation (\ref{eq:lambdaode2}) has an analytic solution $t \mapsto \lambda(t)$ defined on the disk $D_{r(\eta_0)}(0)$ with 
\begin{equation}
\label{eq:r(eta)}
r(\eta_0)  = \int_0^{\sqrt{2}-1} 
f(\sigma) d\sigma.
\end{equation}
Function (\ref{eq:r(eta)}) is represented on Figure \ref{fig:radius}. 
\end{proposition}
\begin{proof}
The right-hand side of (\ref{eq:lambdaode2}) being holomorphic in $\lambda=0$, it has an holomorphic solution $t \mapsto \lambda(t)$ in a neighborhood  of the origin $0 \in \C$. Let $R>0$ be the radius of convergence of the series expansion
\begin{equation}
  \label{eq:lambda(t)}
  \lambda(t) = \sum_{k=1}^\infty \lambda_k\, t^k
\end{equation}
of $\lambda(t)$.
Clearly, $\lambda(\rho)<\sqrt{2}-1$ for $0<\rho<R$; otherwise, it should exist $0<\rho<R$ such that $\lambda(\rho)=\sqrt{2}-1$, which is incompatible with  $\lambda''(\rho) = h(\sqrt{2}-1)$.
Given that the function 
$$
h:  \sigma \mapsto \frac{1 + \sigma}{(1 - 2 \sigma- \sigma^2)^{3/2}}
$$
has a series expansion around $\sigma =0$ with real positive coefficients, we have that
\begin{equation}
  \label{eq:poslambda}
  \forall k\geq 1, \quad \lambda_k\geq 0.
\end{equation}
Hence, 
\begin{equation*}
  \forall n\geq 1, \quad   \lambda^{[n]}(t) := \sum_{k=1}^n \lambda_k\, \rho^k < \sqrt{2}-1
  \end{equation*}
  provided that  $ 0 \leq \rho<R$, and thus $\lambda^{[n]}(R)\leq \sqrt{2}-1$, which implies, thanks to (\ref{eq:poslambda}) that (\ref{eq:lambda(t)}) is convergent for $t=R$ and that $\lambda(R)<\sqrt{2}-1$. In turn, this implies that (\ref{eq:lambda(t)}) is also convergent for
all  $t^* \in \mathcal{B}(R) := \{t \in \C\ : \ |t|=R\}$ and that  
$$
 |\lambda(t^*)| \leq \lambda(R) \leq \sqrt{2}-1.
$$
Actually, $\lambda(R)=\sqrt{2}-1$. Indeed, if  $\lambda(R)<\sqrt{2}-1$, then $h(\sigma)$ is holomorphic for all $\sigma=\lambda(t^*)$,   $t^* \in \mathcal{B}(R)$. Consequently, $\lambda(t)$ is holormorphic in  $\mathcal{B}(R)$ and hence in the closed disk $ \{t \in \C\ : \ |t|\leq R\}$, which contradicts the assumption that $R$ is the radius of convergence of (\ref{eq:lambda(t)}) (i.e., the distance to the nearest singularity of $\lambda(t)$).

If $\eta_0>0$,  multiplying both sides of equation (\ref{eq:lambdaode2}) by $2 \lambda'$ and applying the operator $\int$ on both sides we obtain the first-order differential equation 
\begin{align} \label{eq:lambdafirst}
 f(\lambda) \, \lambda' =1, \quad \lambda(0)=0,
\end{align}
so that, for $0\leq\rho<R$, $F(\lambda(\rho))=\rho$, where
$$
F(\lambda) = \int_0^\lambda f(\sigma) d\sigma.
$$
By continuity of $F$, we finally have that $R=F(\lambda(R))=F(\sqrt{2}-1)$.
\end{proof}

\begin{figure}[t]
\begin{center}
\resizebox{25em}{!}{\includegraphics{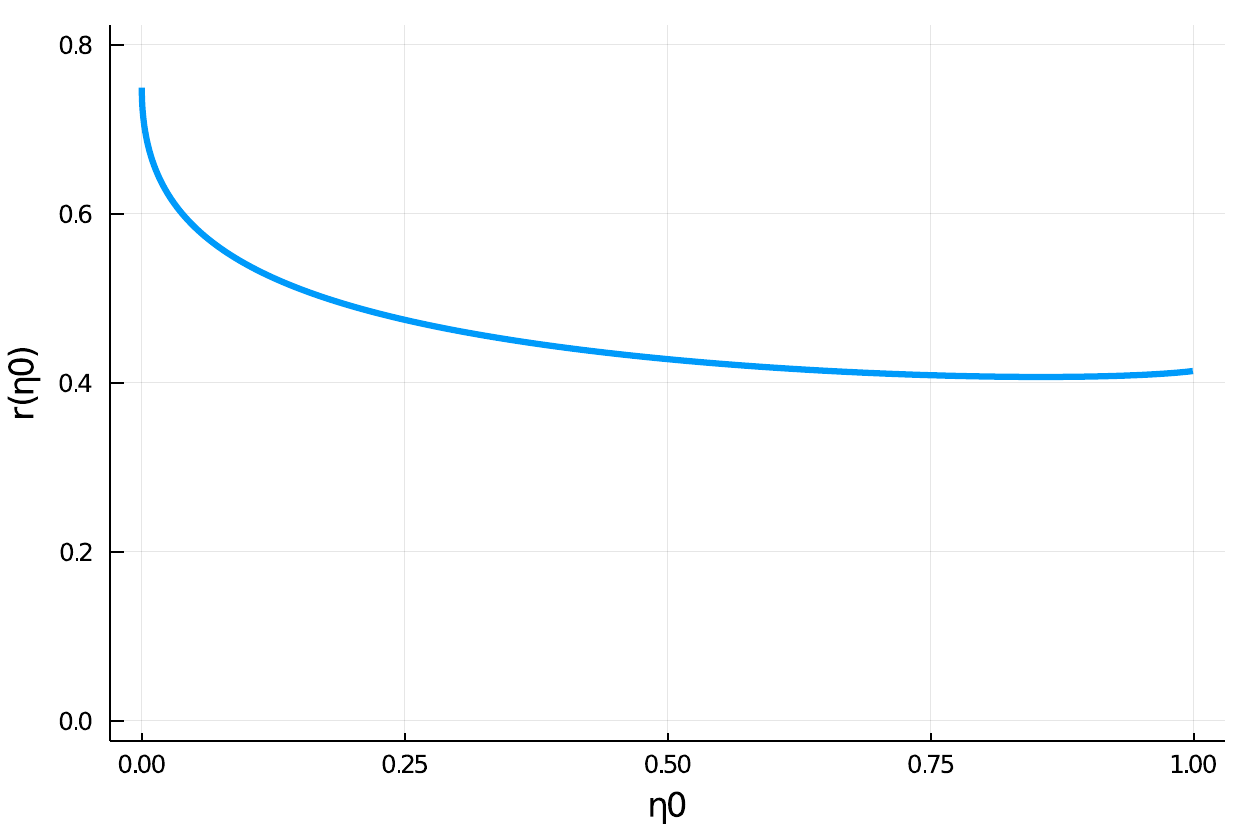}}
\end{center}
\caption{Radius of convergence  $r(\eta_0)$  as a function of $\eta_0\in [0,1)$}
\label{fig:radius}
\end{figure}  
In view of Theorem~\ref{th:phys}, we conclude that the power series expansion
$$
(q,v) = (q_1,\ldots,q_N,v_1,\ldots,v_N) \in \R^{6N}[[t]]
$$
of the solution of  (\ref{eq:Nbody})--(\ref{eq:icond}) with (\ref{eq:icondcond}) is convergent for all $t \in \R$ such that 
$$
|t| < \frac{r(\eta_0)}{\sqrt{\mu_0^2 + \nu_0}},
$$
where 
\begin{equation*}
\eta_0 := \frac{\mu_0^2}{\mu_0^2 + \nu_0} \in (0,1].
\end{equation*}

Hence, we get as a corollary of Theorem~\ref{th:phys} the following result:
\begin{corollary}
\label{cor:1}
The solution of  (\ref{eq:Nbody})--(\ref{eq:icond}) with (\ref{eq:icondcond}) admits an holomorphic extension as a function of the complex time $t$ in the disk 
$$
\left\{ t \in \C\ : \ |t| <  \frac{r(\eta_0)}{\sqrt{\mu_0^2 + \nu_0}}\right\}.
$$\end{corollary}

\begin{remark}
\label{rem:2}
In~\cite{antonana2020},  the statement in Corollary~\ref{cor:1} was proven with a different approach, and a disk of different radius, namely
\begin{equation*}
\sup_{0<\sigma<\sqrt{2}-1}
2 \sigma\, \left(\max_{1 \leq i < j \leq N} \left( \frac{\|v^0_i- v^0_j\|}{ \|q^0_i - q^0_j\|} +
\sqrt{ \frac{\|v^0_i- v^0_j\|^2}{ \|q^0_i - q^0_j\|^2}
+ \kappa(\sigma)\,  \frac{M_{ij}(q^0)}{ \|q^0_i- q^0_j\|} } \right)\right)^{-1},
\end{equation*}
where 
\begin{equation*}
\kappa(\sigma) := \frac{2\, \sigma \, (1 + \sigma)}{(1 - 2 \sigma - \sigma^2)^{3/2}}.
\end{equation*}
If the maxima
\begin{equation*}
\mu_0 = \max_{1 \leq i < j \leq N} \frac{\|v^0_i- v^0_j\|}{ \|q^0_i - q^0_j\|} \quad \mbox{and} \quad
\nu_0  =  \max_{1 \leq i < j \leq N}  \frac{M_{ij}(q^0)}{ \|q^0_i- q^0_j\|}
\end{equation*}
are attained at a common pair of indices $(i,j)$ (which is typically the case, specially in binary close encounters), then this coincides with 
\begin{equation*}
\sup_{0<\sigma<\sqrt{2}-1}
2 \sigma\, \left( \mu_0 +
\sqrt{ \mu_0^2
+ \kappa(\sigma)\,  \nu_0 } \right)^{-1} = \frac{\hat r(\eta_0)}{\sqrt{\mu_0^2+\nu_0}}, 
\end{equation*}
where 
\begin{equation*}
\hat r(\eta_0) := \sup_{0<\sigma<\sqrt{2}-1}
2 \sigma\, \left( \sqrt{\eta_0} +
\sqrt{ \eta_0
+ \kappa(\sigma)\,  (1-\eta_0) } \right)^{-1} .
\end{equation*}
One can check that $\hat r(\eta_0) \leq r(\eta_0)$ for all $\eta_0 \in [0,1]$.
For instance,
\begin{itemize}
\item  if  $\eta_0=1/2$ (corresponding to the case $\nu_0=\mu_0^2$), then $\gamma=-1+\sqrt{2}$ and 
\begin{equation*}
r(1/2) = F(\lambda) = \int_0^{\gamma(1/2)} 
\frac{1}{\sqrt{1/2- 1 +\frac{1}{\sqrt{1-2  \sigma - \sigma^2}}}} d\sigma \approx 0.42812819,
\end{equation*}
while 
\begin{equation*}
\hat r(1/2) = \frac{2\sqrt{2} \sigma}{1 + \sqrt{1 + \kappa(\sigma)}} \approx 0.25796556.
\end{equation*}

\item if  $\eta_0\to 1$,  then $\gamma(\eta_0) \to -1+\sqrt{2}$ and 
\begin{equation*}
r(\eta_0) \to \int_0^{\gamma(\eta_0)}
1 d\sigma = \gamma(\eta_0) = -1+\sqrt{2},
\end{equation*}
which coincides with
\begin{equation*}
\hat r(0) = \sup_{0<\sigma<\sqrt{2}-1} \sigma= -1+\sqrt{2}.
\end{equation*}
\end{itemize}
\end{remark}

\begin{remark}
\label{rem:fit}
We have numerically checked that $0.407 \leq r(\eta) \leq 0.75$ for all $\eta\in[0,1]$, and thus, the following lower and upper  bounds
\begin{equation*}
\frac{0.407}{\sqrt{\mu_0^2 + \nu_0}} < \frac{r(\eta_0)}{\sqrt{\mu_0^2 + \nu_0}} <  \frac{0.75}{\sqrt{\mu_0^2 + \nu_0}}.
\end{equation*}
of the estimate of the radius of convergence given in Corollary~\ref{cor:1} hold.
\end{remark}

\section{The $N$-body problem: Time-renormalized equations}
\label{sect:Nbptre}

Consider an $N$-body problem described with time-renormalized equations
\begin{equation}
\label{eq:Nbodytau}
\begin{array}{lll}
 \displaystyle    \frac{d Q_i}{d \tau} &= s(Q,V)\, V_i, & i=1, \ldots, N, \\
     && \\
 \displaystyle \frac{d V_i}{d \tau} &= \displaystyle s(Q,V)\,  g_i(Q) ,& i=1, \ldots, N,
\end{array}
\end{equation}
where 
$$
g_i(Q) = \sum_{j\neq i} \frac{G\, m_j}{\|Q_i - Q_j\|^{3}}(Q_j-Q_i).
$$
For smooth time-renormalization functions $s(Q,V)$, the solution of (\ref{eq:Nbodytau}) supplemented with initial conditions
\begin{equation}
\label{eq:icondtau}
Q_i(0) = q_i^0, \quad V_i(0) = v_i^0, \quad i=1,\ldots,N,
\end{equation}
admits a unique formal power series expansion ($Q_i \in \R^3[[\tau]]$, $V_i \in \R^3[[\tau]]$, $i=1,\ldots,N$) provided that (\ref{eq:icondcond}) holds.  In what follows, we will always assume that the initial state values (\ref{eq:icondtau}) are non-singular, that is, they satisfy the regularity condition (\ref{eq:icondcond}).

We want to choose a real analytic function $s(Q,V)$ in such a way that the radius of convergence of such power series expansions are uniformly bounded from below by a positive constant $\beta$ for all regular initial values and all values of the masses.  This is motivated by the following observation: assume that one wants to discretize the time-renormalized equations with constant fictitious time-step $\Delta \tau$. This is expected to give poor accuracy if the infimum of the radius of convergence along the solution trajectory is smaller than $|\Delta\tau|$. The sought uniform bound from below $\beta$ of the radius of convergence guarantees that this will not happen provided that $\Delta \tau < \beta$.

Based on a heuristic argument, we will choose $s(Q,V)$ as a real-analytic function such that $s(q^0,v^0)$ is, up to a constant factor, a lower bound of the radius of convergence  of the  series expansion in powers of the physical time $t$ of the solution $(q(t),v(t))$ of the $N$-body problem (\ref{eq:Nbody})--(\ref{eq:icond}). We will first choose such a function, and then we will prove that a uniform bound from below $\beta>0$  of the radius of convergence can be obtained for that particular time-renormalized function.

Based on Corollary~\ref{cor:1} and Remark~\ref{rem:fit},  we determine the time-renormalization function in such a way that
$s(q^0,v^0) \leq (\mu_0^2 + \nu_0)^{-1/2}$, that is,
\begin{equation*}
s(Q,V) \leq \left(\max_{1 \leq i < j \leq N} \left(\frac{\|V_i- V_j\|}{ \|Q_i - Q_j\|}\right)^2 +
\ \max_{1 \leq i < j \leq N}  \frac{M_{ij}(Q)}{ \|Q_i- Q_j\|}\right)^{-1/2}. 
\end{equation*}
Since 
\begin{align*}
\max_{1 \leq i < j \leq N} \left(\frac{\|V_i- V_j\|}{ \|Q_i - Q_j\|}\right)^2 &\leq 
\sum_{1 \leq i < j \leq N} \frac{\|V_i - V_j\|^{2}}{\|Q_i - Q_j\|^{2}}, \\
\max_{1 \leq i < j \leq N}  \frac{M_{ij}(Q)}{ \|Q_i- Q_j\|} &\leq 
\sum_{1 \leq i < j \leq N}
 \frac{M_{ij}(Q)}{\|Q_i - Q_j\|}, 
\end{align*}
the following real analytic function
\begin{equation}
\label{eq:s}
s(Q,V) = \left(\sum_{1 \leq i < j \leq N}
         \frac{\|V_i - V_j\|^{2}}{\|Q_i - Q_j\|^{2}} + \sum_{1 \leq i < j \leq N}
 \frac{M_{ij}(Q)}{\|Q_i - Q_j\|} \right)^{-1/2}
\end{equation}
satisfies the required inequality above.
The time-renormalization function (\ref{eq:s}) was first proposed in~\cite{antonana2020}, and it was proven that the required uniform lower bound of the radius of convergence of the corresponding time-renormalized equations holds for that renormalization function. The rest of the present section will be devoted to prove that result with the technique of majorant equations. With the new proof,  a tighter lower bound $\beta$ is obtained.

\begin{lemma}
\label{lem:3}  
Consider the power series expansion 
$$
(Q,V) = (Q_1,\ldots,Q_N,V_1,\ldots,V_N) \in \R^{6N}[[\tau]]
$$
of the solution of  (\ref{eq:Nbodytau})--(\ref{eq:s}). If the power series 
 \begin{equation}
  \label{eq:xizeta} 
\xi = 1 + \sum_{k\geq 1} \xi_k\, \tau^k  \in  \R_+[[\tau]], \quad
\zeta = \sum_{k\geq 1} \zeta_k\, \tau^k  \in  \R_+[[\tau]]
\end{equation}
are such that, for $1 \leq i < j \leq N$
\begin{align}
\label{eq:Qijtleq}
  Q_i - Q_j &\tleq \|q_i^0 - q_j^0\|\, \xi, \\
  \label{eq:Vijtleq}
  V_i - V_j &\tleq \|v_i^0 - v_j^0\| + s(q^0,v^0) \, M_{ij}(q^0)\,  \zeta,
\end{align}
then, for $1 \leq i < j \leq N$
\begin{align*}
  g_i(Q) \tleq
 K_i(q^0)
  \frac{\xi}{(2-\xi^2)^{3/2}}.
\end{align*}
and 
\begin{align*}
s(Q,V) \tleq s(q^0,v^0) \left( 2 - \chi(\xi,\zeta) \right)^{-1/2},
\end{align*}
where 
\begin{equation}
\label{eq:chi}
\chi(\xi,\zeta) = (2-\xi^2)^{-1} \, (2 \zeta + \zeta^2 +(2-\xi^2)^{-1/2}).
\end{equation}
\end{lemma}
\begin{proof}
Proceeding as in the proof of Lemma \ref{lem:phys}, we obtain 
\begin{align} \label{eq:Q2}
\forall 1 \leq i < j \leq N, \quad \frac{1}{\|Q_i -Q_j\|^2} \tleq \frac{1}{\|q_i^0 -q_j^0\|^2} \, (2-\xi^2)^{-1} ;
\end{align}
then 
\begin{align*}
\forall 1 \leq i < j \leq N, \quad \frac{1}{\|Q_i -Q_j\|^3} \tleq \frac{1}{\|q_i^0 -q_j^0\|^3} \, (2-\xi^2)^{-3/2}
\end{align*}
 and eventually, for $1 \leq i \leq N$,
\begin{equation*}
 g_i(Q) \tleq
 K_i(q^0)
  \frac{\xi}{(2-\xi^2)^{3/2}}.
\end{equation*}
Similarly, we get
\begin{align} \label{eq:bis}
\forall 1 \leq i < j \leq N, \quad \frac{1}{\|Q_i -Q_j\|} \tleq \frac{1}{\|q_i^0 -q_j^0\|} \, (2-\xi^2)^{-1/2}
\end{align}
so that by summing over $j$ and using relation (\ref{eq:add}) of Proposition \ref{prop:majo}, we get 
\begin{equation} \label{eq:K}
K_i(Q) \tleq  K_i(q^0) \, (2-\xi^2)^{-1}.
\end{equation}
Now, from  assumption (\ref{eq:Vijtleq}) and inequality (\ref{eq:scal}) of Proposition \ref{prop:majo}, we have 
\begin{align} \label{eq:estV}
\|V_i-V_j\|^2 &\tleq \left(\|v_i^0 - v_j^0\| + s(q^0,v^0) \, M_{ij}(q^0)\,  \zeta\right)^2  \nonumber \\
&\tleq \|v_i^0 - v_j^0\|^2 +  \|q_i^0 - q_j^0\|\, M_{ij}(q^0)\,  (2 \zeta + \zeta^2),
\end{align}
where we have further applied the inequalities 
\begin{equation}
  \label{eq:ss^2}
s(q^0,v^0) \leq \frac{\|q_i^0-q_j^0\|}{\|v_i^0-v_j^0\|} \quad \mbox{ and } \quad s(q^0,v^0)^2 \leq \frac{\|q_i^0-q_j^0\|}{M_{i j}(q^0)}.
\end{equation}
Combining (\ref{eq:Q2}) and (\ref{eq:estV}) through inequality (\ref{eq:prod}) of Proposition \ref{prop:majo}, we get 
\begin{align*}
\frac{\|V_i-V_j\|^2}{\|Q_i-Q_j\|^2} &\tleq 
\frac{\|v_i^0 - v_j^0\|^2}{ \|q_i^0 - q_j^0\|^2} \, (2-\xi^2)^{-1}+  \frac{M_{i j}(q^0)}{ \|q_i^0 - q_j^0\|}\,  (2 \zeta + \zeta^2) \, (2-\xi^2)^{-1}.
\end{align*}
Similarly, combining (\ref{eq:bis}) and (\ref{eq:K}), we obtain
\begin{align*}
\frac{M_{ij}(Q)}{\|Q_i -Q_j\|} &\tleq \frac{M_{ij}(q^0)}{\|q_i^0-q_j^0\|}  \, (2-\xi^2)^{-1}  \, (2-\xi^2)^{-1/2}.
\end{align*}
Finally, consider 
\begin{align*}
s_A(Q,V) &= 
\sum_{1 \leq i < j \leq N}
         \frac{\|V_i - V_j\|^{2}}{\|Q_i - Q_j\|^{2}}, \\
  s_B(Q,V) &= 
 \sum_{1 \leq i < j \leq N}
 \frac{M_{ij}(Q)}{\|Q_i - Q_j\|}.
\end{align*}
so that
\begin{equation*}
s(Q,V) = \left(s_A(Q,V) + s_B(Q,V) \right)^{-1/2}.
\end{equation*}
Then,
\begin{align*}
s_A(Q,V) +s_B(Q) &\tleq
 \left(s_A(q^0,v^0)   + s_B(q^0)\, (2 \zeta + \zeta^2 +(2-\xi^2)^{-1/2} ) \right) \, (2-\xi^2)^{-1}, \\
 &\leq  \left(s_A(q^0,v^0)   + s_B(q^0) \right)  \, \chi(\xi,\zeta),
\end{align*}
where $\chi(\xi,\zeta)$ is given 
in terms of $\xi$ and $\zeta$
by (\ref{eq:chi}), and it follows from (\ref{eq:ffb}) with $\nu=-1/2$ that 
\begin{equation*}
s(Q,V) \tleq s(q^0,v^0) \left( 2 - \chi(\xi,\zeta) \right)^{-1/2}.
\end{equation*}
\end{proof}
\begin{lemma}
  \label{lem:4}
Let $\mathcal{B}$ be the set of $(\xi,\zeta) \in   (1 + \tau \R[[\tau]]) \times \tau \R[[\tau]]$ 
satisfying (\ref{eq:Qijtleq})--(\ref{eq:Vijtleq}) (with the function $s(Q,V)$ given in (\ref{eq:s})), and consider the operator
\begin{equation}
\label{eq:Psi}
\begin{array}{crcl}
\Psi:& (1 + \tau \R[[\tau]]) \times \tau \R[[\tau]] &\rightarrow&  \R[[\tau]] \times \R[[\tau]]\\
  & (\xi,\zeta)  &\mapsto&  \left(1 + \int (2-\chi(\xi,\zeta))^{-1/2}  \, (1+\zeta),
                          \int  \frac{ (2-\chi(\xi,\zeta))^{-1/2} \, \xi}{(2-\xi^2)^{3/2}} \right).
\end{array}
\end{equation}
Then $\Psi(\mathcal{B}) \subset \mathcal{B}$.
\end{lemma} 
\begin{proof}
Recall that
\begin{align}
\label{eq:Qi-Qj}  
Q_i - Q_j &= q_i^0 - q_j^0 + \int s(Q,V) \, (V_i - V_j), \\
\label{eq:Vi-Vj}  
V_i -V_j &= v_i^0 - v_j^0 + \int s(Q,V) \, \left(g_i(Q)-g_j(Q)\right).
\end{align}
If $(\xi,\zeta)\in \mathcal{B}$, then (\ref{eq:Qi-Qj}), Lemma~\ref{lem:3}, and (\ref{eq:ss^2}) imply that, for
$1 \leq i < j \leq N$
\begin{align*}
Q_i - Q_j &\tleq \|q_i^0 - q_j^0\| + s(q^0,v^0) \,  \|v_i^0-v_j^0\| \, \int (2-\chi(\xi,\zeta))^{-1/2} \\
 & + s(q^0,v^0)^2 \, M_{ij}(q^0)\,  \int (2-\chi(\xi,\zeta))^{-1/2} \, \zeta
  \\
  &\tleq \|q_i^0 - q_j^0\| \left(1 + \int (2-\chi(\xi,\zeta))^{-1/2}  \, (1+\zeta) \right).
\end{align*}
Similarly,  (\ref{eq:Vi-Vj}) and Lemma~\ref{lem:3} imply that
\begin{align*}
 V_i - V_j &\tleq \|v_i^0 - v_j^0\| + s(q^0,v^0)\,  M_{ij}(q^0) \, \int  \frac{ (2-\chi(\xi,\zeta))^{-1/2} \, \xi}{(2-\xi^2)^{3/2}}.
  \end{align*}
  Hence, we conclude that $\Psi(\xi,\zeta) \in \mathcal{B}$. 
\end{proof}

\begin{theorem}
\label{th:reg}
  The power series representation 
 $$
 (Q,V) = (Q_1,\ldots,Q_N,V_1,\ldots,V_N) \in \R^{6N}[[\tau]]
 $$
 of the solution  of equations (\ref{eq:Nbodytau})--(\ref{eq:s})  satisfies (\ref{eq:Qijtleq})--(\ref{eq:Vijtleq}), where
  $(\xi,\zeta)  \in  \R_+[[\tau]] \times  \R_+[[\tau]]$ is the power series solution of the following initial value problem
  \begin{equation}
  \label{eq:xizetaode}
    \begin{split}
       \xi' &=   (1+\zeta)\, (2-\chi(\xi,\zeta))^{-1/2}, \quad \xi(0)=1,\\
       \zeta'  &=  
      \frac{\xi\,  (2-\chi(\xi,\zeta))^{-1/2}}{(2-\xi^2)^{3/2}}, \quad \zeta(0)=0,
    \end{split}  
\end{equation}
where $\chi(\xi,\zeta)$ is given as a function of $(\xi,\zeta)$ by (\ref{eq:chi}).
\end{theorem}
\begin{proof}
Our proof of Theorem~\ref{th:reg} mimics the proof of Theorem~\ref{th:phys}. We begin by considering the sequence $\{\xi^{[m]},\zeta^{[m]}\}_{m\in \N}$, where $(\xi^{[m]},\zeta^{[m]}) = \Psi(\xi^{[m-1]},\zeta^{[m-1]})$ for $m\geq 1$
and 
$$
\xi^{[0]} = 1 + \sum_{k\geq 1} \xi^{[0]} _k\, \tau^k  \in  \R_+[[\tau]], \quad
\zeta^{[0]}  = \sum_{k\geq 1} \zeta^{[0]} _k\, \tau^k  \in  \R_+[[\tau]]
$$ 
is such that for $1 \leq i < j \leq N$ and $m\geq 1$, 
\begin{equation*}
\xi^{[0]} _m =  \frac{\|(Q_i - Q_j)_m\|}{\|q_i^0 - q_j^0\|}  \quad \mbox{ and } \quad 
\zeta^{[0]} _m =   \frac{\|(V_i - V_j)_m\|}{s(q^0,v^0) \, M_{ij}(q^0)}.
\end{equation*}
Clearly, $(\xi^{[0]} ,\zeta^{[0]} ) \in \mathcal{B}$, so that by Lemma~\ref{lem:4},
$(\xi^{[m]} ,\zeta^{[m]}) \in \mathcal{B}$
for $m\geq 1$.  Proceeding as in  the proof of Theorem \ref{th:phys}, one concludes that the sequence 
$\{\xi^{[m]},\zeta^{[m]}\}_{m\in \N}$ converges (in the sense of each coefficient of the two series are ultimately constant) towards a limit $(\xi^{[\infty]},\zeta^{[\infty]}) \in \mathcal{B}$, which is the unique solution of the fixed point equation 
$$
(\xi^{[\infty]},\zeta^{[\infty]}) =\Psi(\xi^{[\infty]},\zeta^{[\infty]}).
$$
We thus have that estimates (\ref{eq:Qijtleq})--(\ref{eq:Vijtleq}) hold for $(\xi,\zeta) =(\xi^{[\infty]},\zeta^{[\infty]})$ the solution of
\begin{equation*}
\left\{
\begin{array}{rcl}
\xi &=& 1 + \int (2-\chi(\xi,\zeta))^{-1/2}  \, (1+\zeta), \\
\zeta &=& \int  \frac{ (2-\chi(\xi,\zeta))^{-1/2} \, \xi}{(2-\xi^2)^{3/2}} 
\end{array}
\right.
\end{equation*}
(where $\chi(\xi,\zeta)$ is given in terms of $\xi$ and $\zeta$ by (\ref{eq:chi})), 
or in other words, the unique power series solution of (\ref{eq:xizetaode}). 
\end{proof}

\begin{theorem}
\label{th:reg2}
Under the assumptions of Theorem~\ref{th:reg}, for $1\leq i \leq N$,
\begin{align*}
V_i - v_i^0 & \tleq s(q^0,v^0) \, K_i(q^0) \, \zeta, \\
Q_i - q_i^0 & \tleq \max\left(s(q^0,v^0) \, \|v_i^0\|,  s(q^0,v^0)^2 \, K_i(q^0) \right)  \, (\xi-1) . 
\end{align*}\end{theorem}
\begin{proof}
The following majorants for $Q_i -q_i^0$ and $V_i -v_i^0$ can be obtained from  $Q_i - q_i^0 = \int s(Q,V) V_i$ and $V_i - v_i^0 = \int s(Q,V) g_i(Q)$ 
respectively by virtue of Lemma~\ref{lem:3},
\begin{align*}
V_i - v_i^0 & \tleq s(q^0,v^0) \, K_i(q^0) \, \int  \frac{\xi\,  (2 - \chi(\xi,\zeta))^{-1/2}}{(2-\xi^2)^{3/2}} = s(q^0,v^0) \, K_i(q^0) \, \zeta, \\
Q_i - q_i^0  & \tleq s(q^0,v^0) \, \|v_i^0\| \, \int (2 - \chi(\xi,\zeta))^{-1/2}  +  s(q^0,v^0)^2 \, K_i(q^0)  \,  \int \zeta \, (2 - \chi(\xi,\zeta))^{-1/2} \\
& \tleq \max\left(s(q^0,v^0) \, \|v_i^0\|,  s(q^0,v^0)^2 \, K_i(q^0) \right)  \, (\xi-1) . 
\end{align*}
\end{proof}

\begin{proposition}
\label{prop:R}
 The radius of convergence $R$ of the power series solution $(\xi, \zeta)$ of (\ref{eq:xizetaode}) is given by 
\begin{align}
 \label{eq:R}
R=G(\xi) = \int_0^{v_+} g(\sigma) d\sigma \approx  0.0839968103939379,
\end{align}
where 
$$
g(\sigma) = 2\,{\frac {1}{ \left( {\sigma}^{2}+2\,\sigma+2 \right) ^{2}}\sqrt {-
\,{\frac {3\,{\sigma}^{6}+18\,{\sigma}^{5}+50\,{\sigma}^{4}+80\,{\sigma}^{
3}+76\,{\sigma}^{2}+40\,\sigma-8}{{\sigma}^{4}+4\,{\sigma}^{3}+8\,{\sigma}^
{2}+8\,\sigma+2}}}}
$$
and 
\begin{align*}
v_{+} &=-1 +  {\frac { \sqrt {502+18\,\sqrt {777}-5
\, \left( 251+9\,\sqrt {777} \right) ^{2/3}+8\,\sqrt [3]{251+9\,\sqrt 
{777}}}}{3\,\sqrt [3]{251+9\,\sqrt {777}}}}\\
&\approx 0.149902575567304.
\end{align*}
\end{proposition}
\begin{proof}
  The solution $(\xi, \zeta)$ of (\ref{eq:xizetaode}), where  $\chi=\chi(\xi,\zeta)$ is given by (\ref{eq:chi}),
can  be computed alternatively as follows: obtain $\zeta$ as the initial value problem
\begin{equation} \label{eq:zeta}
      \zeta'  =  
       (1 + \gamma)^2 \sqrt{\frac{1 + 4\gamma + 2 \gamma^2}{2-\chi}}, \quad \zeta(0)=0,  
\end{equation}
where $\gamma =\zeta + \zeta^2/2$ and $\chi = (1+\gamma)^2 (1 + 3\, \gamma)$, and then
\begin{equation}
  \label{eq:xiaux}
\xi = \frac{\sqrt{1 + 4\,\gamma + 2 \, \gamma^2} }{\gamma + 1}.
\end{equation}
Indeed, it is straightforward to check that $ (2 - \xi^2)^{-1/2} = \gamma + 1$ holds for the solution of (\ref{eq:xizetaode}), which implies that  $\chi = (1+\gamma)^2 (1 + 3\, \gamma)$ and
\begin{equation*}
  \xi = \sqrt{2 - \left(\gamma+ 1 \right)^{-2}} = \frac{\sqrt{1 + 4\,\gamma + 2 \, \gamma^2} }{\gamma + 1}.
\end{equation*}
Being majorant series by construction, both $\zeta$ and $\xi$ have expansions in powers of $t$ with real positive coefficients. Using the same argument for $\zeta$ as for $\lambda$ in Proposition \ref{prop:r}, we can show that equation (\ref{eq:zeta}) has an analytic solution $\zeta(\tau)$ on the disk $D_R(0)$, where 
\begin{align*}
R &= \int_0^{v_+} g(\sigma) d\sigma, \\
g(\sigma) &= 2\,{\frac {1}{ \left( {\sigma}^{2}+2\,\sigma+2 \right) ^{2}}\sqrt {-
\,{\frac {3\,{\sigma}^{6}+18\,{\sigma}^{5}+50\,{\sigma}^{4}+80\,{\sigma}^{
3}+76\,{\sigma}^{2}+40\,\sigma-8}{{\sigma}^{4}+4\,{\sigma}^{3}+8\,{\sigma}^
{2}+8\,\sigma+2}}}},
\end{align*}
and 
$$
v_+ := \sup_{\tau \in D_R(0)} |\zeta(\tau)|
$$
is the root of $\sigma^4+4 \sigma^3 + 8 \sigma^2 + 8 \sigma+2$ with smallest modulus. As the right-hand side of equation (\ref{eq:xiaux}) is also analytic on $D_{v_+}(0)$ as a function of $\xi$, the other component $\xi(t)$ of the solution of equation (\ref{eq:xizetaode}) is well-defined and analytic on the same disk $D_R(0)$. 
\end{proof}

In view of Theorem~\ref{th:phys}, we conclude that the power series expansion
$$
(Q,V) = (Q_1,\ldots,Q_N,V_1,\ldots,V_N) \in \R^{6N}[[\tau]]
$$
of the solution of  (\ref{eq:Nbodytau})--(\ref{eq:icondtau}) with (\ref{eq:icondcond}) is convergent for all $\tau \in (-R,R)$. Hence, we get as a corollary of Theorem~\ref{th:reg} 
the following result, originally proven in~\cite{antonana2020} with $R =  0.0444443$.

\begin{corollary}
\label{cor:strip}
The solution of  (\ref{eq:Nbodytau})--(\ref{eq:icondtau}) with (\ref{eq:icondcond}) admits an holomorphic extension as a function of the complex time $\tau$ in the strip 
\begin{equation}
\label{eq:strip}
\{ \tau \in \C\ : \ |\mathrm{Im}(\tau)| <  \beta=0.0839968103939379\}. 
\end{equation}
\end{corollary}

\begin{remark}
As pointed out in~\cite{antonana2020}, Corollary~\ref{cor:strip} implies that the solution of  (\ref{eq:Nbodytau}) with regular initial values admits a globally convergent series expansion in powers of a new variable $\sigma$, related to $\tau$ with
the conformal mapping
\begin{equation*}
\tau \mapsto \sigma =\frac{\exp(\frac{\pi}{2\beta} \tau)-1}{\exp(\frac{\pi}{2\beta} \tau)+1}.
\end{equation*}
that maps the strip (\ref{eq:strip}) into the unit disk. 
This is closely related to Sundman's result~\cite{sundman} for the 3-body problem as well as Wang's results~\cite{wang} for the general case of $N$-body problems. It is worth emphasizing that, in contrast with both Sundman's and Wang's solutions,  our approach remains valid in the limit where  $\displaystyle \min_{1\leq i  \leq N} m_i/M \to 0$ with $M=\sum_{1 \leq i \leq N} m_i$.
\end{remark}

\section{Discretization of the time-renormalized $N$-body equations}
\label{sect:DottrNbe}

We now consider the implicit mid-point rule discretization of the equations (\ref{eq:Nbodytau}). 
 The implicit midpoint rule gives, for small enough values of the step-zie $h$ in $\tau$, an approximation of the solution  of the initial value problem (\ref{eq:Nbodytau})--(\ref{eq:icondtau}).  We want to obtain majorants of the power series expansions in powers of the step-size $h$  of the local errors of the implicit midpoint approximation
of the solution of (\ref{eq:Nbodytau})--(\ref{eq:icondtau}). The main goal of the present section is to show that, with the time-renormalization function (\ref{eq:s}), it makes sense discretizing the time-renormalized equations with constant time-steps, without drastically degrading the accuracy during close encounters.

In the present section, it is always assumed that $s(Q,V)$ is given by (\ref{eq:s}).
Let  
 $$
 (\tilde Q, \tilde V) = (\tilde Q_1,\ldots, \tilde Q_N,\tilde V_1,\ldots, \tilde V_N) \in \R^{6N}[[\Delta\tau]]
 $$
 be the power series expansion of the implicit midpoint  approximation, and consider $(\hat Q, \hat V) = \frac12 \, (q^0 + \tilde Q, v^0 + \hat V)$. Then, for $1\leq i \leq N$ it holds  that
 \begin{equation}
 \label{eq:hatQV}
 \begin{split}
\hat Q_i &= q_i^0 + \frac{\Delta\tau}{2} \, s(\hat Q, \hat V) \hat V_i, \\
\hat V_i &= v_i^0 + \frac{\Delta\tau}{2} \, s(\hat Q, \hat V) g_i(\hat Q), 
\end{split}
\end{equation}
and
 \begin{equation}
 \label{eq:tildeQV}
 \begin{split}
\tilde Q_i &= q_i^0 + \Delta\tau\, s(\hat Q, \hat V) \hat V_i, \\
\tilde V_i &= v_i^0 + \Delta\tau \, s(\hat Q, \hat V) g_i(\hat Q).
\end{split}
\end{equation}

\begin{lemma}
  \label{lem:5}
Let $\hat{\mathcal{B}}$ be the set of $(\xi,\zeta) \in   (1 + \Delta\tau \R[[\Delta\tau]]) \times \Delta\tau \R[[\Delta\tau]]$ such that
for $1 \leq i < j \leq N$
\begin{align*}
 \hat  Q_i -\hat  Q_j &\tleq \|q_i^0 - q_j^0\|\, \xi, \\
\hat  V_i - \hat V_j &\tleq \|v_i^0 - v_j^0\| + s(q^0,v^0) \, M_{ij}(q^0)\,  \zeta,
\end{align*}
and consider the operator
\begin{equation*}
\begin{array}{crcl}
\widehat \Psi:&   (1 + \Delta\tau \R[[\Delta\tau]]) \times \Delta\tau \R[[\Delta\tau]] &\rightarrow&  \R[[\Delta\tau]] \times \R[[\Delta\tau]]\\
  & (\xi,\zeta)  &\mapsto&  \left(1 + \frac{\Delta\tau}{2}\,  (2-\chi(\xi,\zeta))^{-1/2}  \, (1+\zeta),
                         \frac{\Delta\tau}{2}\,  \frac{ (2-\chi(\xi,\zeta))^{-1/2} \, \xi}{(2-\xi^2)^{3/2}} \right).
\end{array}
\end{equation*}
Then $\widehat{\Psi}(\hat{\mathcal{B}}) \subset  \hat{\mathcal{B}}$.
\end{lemma} 
\begin{proof} We only sketch the proof as it is analogous to that of Lemma~\ref{lem:4}. Expressing relative positions and velocities between bodies, we have for all $1 \leq i < j \leq N$
\begin{align*}
\hat Q_i - \hat Q_j & = q_i^0 - q_j^0 + \frac{\Delta\tau}{2} \, s(\hat Q, \hat V) \left( \hat V_i - V_j \right), \\
\hat V_i - \hat V_j &= v_i^0 - v_j^0 + \frac{\Delta\tau}{2} \, s(\hat Q, \hat V) \left(g_i(\hat Q)-g_j(\hat Q)\right)
\end{align*}
so that, upon using Lemma \ref{lem:3} with $(Q,V)$ replaced by $(\hat Q, \hat V)$, we immediately obtain 
\begin{align*}
\hat Q_i - \hat Q_j &  \tleq \|q_i^0 - q_j^0\| + \frac{\Delta\tau}{2} \, s(q^0, v^0) (2-\chi(\xi,\zeta))^{-1/2} \Big(  \|v_i^0 - v_j^0\| + s(q^0, v^0) M_{ij}(q^0)\Big), \\
\hat V_i - \hat V_j & \tleq \|v_i^0 - v_j^0\| + \frac{\Delta\tau}{2} \, s(q^0, v^0) (2-\chi(\xi,\zeta))^{-1/2} \left(K_i(q^0) +K_j(q^0) ) \right) \frac{\xi}{(2-\xi^2)^{1/2}}.
\end{align*}
It then follows from the bounds in (\ref{eq:ss^2}) that 
\begin{align*}
\hat Q_i - \hat Q_j &  \tleq \|q_i^0 - q_j^0\| \left( 1+ \frac{\Delta\tau}{2} \,  (2-\chi(\xi,\zeta))^{-1/2} (1+\xi) \right), \\
\hat V_i - \hat V_j & \tleq \|v_i^0 - v_j^0\| + s(q^0, v^0) M_{ij}(q^0)\left( \frac{\Delta\tau}{2} \, (2-\chi(\xi,\zeta))^{-1/2} \frac{\xi}{(2-\xi^2)^{1/2}} \right).
\end{align*}
This proves that $\hat \Psi (\mathcal{B}) \subset \mathcal{B}$.
\end{proof}

\begin{theorem}
\label{th:hatreg} 
For $1\leq i < j \leq N$, 
\begin{align}
\label{eq:hatQijtleq}
 \hat  Q_i -\hat  Q_j &\tleq \|q_i^0 - q_j^0\|\, \hat \xi, \\
  \label{eq:hatVijtleq}
\hat  V_i - \hat V_j &\tleq \|v_i^0 - v_j^0\| + s(q^0,v^0) \, M_{ij}(q^0)\,  \hat \zeta,
\end{align}
where $(\hat \xi, \hat \zeta) \in  (1+\R[[\Delta\tau]]) \times \Delta\tau \R[[\Delta\tau]]$ is the unique fixed point of $\widehat{\Psi}$.
\end{theorem}
\begin{proof}
Again, we only sketch the proof, which is similar to the proof of Theorem \ref{th:reg}. The operator $\hat \Psi$ being clearly Noetherian, the sequence 
$$
(\xi^{[m]},\zeta^{[m]}) = \hat \Psi \left( \xi^{[m-1]},\zeta^{[m-1]}\right), \quad m=1, \ldots, 
$$
where
$$
\xi^{[0]} = 1 + \sum_{k\geq 1} \xi^{[0]} _k\, \Delta\tau^k  \in  \R_+[[\Delta\tau]] \quad \mbox{ and } \quad 
\zeta^{[0]}  = \sum_{k\geq 1} \zeta^{[0]} _k\, \Delta\tau^k  \in  \R_+[[\Delta\tau]]
$$ 
with 
\begin{equation*}
\xi^{[0]} _m =   \frac{\|(\hat Q_i - \hat Q_j)_m\|}{\|q_i^0 -  q_j^0\|} \quad \mbox{ and } \quad 
\zeta^{[0]} _m =   \frac{\|(\hat V_i - \hat V_j)_m\|}{s(q^0,v^0) \, M_{ij}(q^0)},
\end{equation*}
converges in $\mathcal{B}$ to a limit $(\xi^{[\infty]},\zeta^{[\infty]})$ (owing to previous lemma). This limit is the unique series in powers of $\Delta\tau$ satisfying  the equations
\begin{align}
\hat \xi &= 1 + \frac{\Delta\tau}{2}\,  (2-\chi(\hat\xi,\hat\zeta))^{-1/2}  \, (1+\hat\zeta), \label{eq:hxi}\\
\hat\zeta & = \frac{\Delta\tau}{2}\,  \frac{ (2-\chi(\hat \xi,\hat \zeta))^{-1/2} \, \hat \xi}{(2-\hat \xi^2)^{3/2}},
\end{align}
satisfying $\left.(\hat \xi(\Delta\tau),\hat \zeta(\Delta\tau))\right|_{\Delta\tau=0} = (1,0)$.
\end{proof}
\begin{remark}
It is not difficult to obtain the following relation
$$
\hat \zeta =\frac12 \left( \left( 1-\frac{4 \hat \xi (1-\hat\xi)}{(2-\hat \xi^2)^{3/2}} \right)^{1/2}-1\right)
$$
which in turn, can be substituted into (\ref{eq:hxi}) for instance, to obtain an algebraic equation involving only $\hat \xi$ and $\Delta\tau$. It is then possible to solve this equation numerically in order to estimate the radius of convergence of the series $\hat \xi(\Delta\tau)$. 
\end{remark}
Next theorem can be proven along the same lines as Theorem~\ref{th:reg2}.
\begin{theorem}
\label{th:reg_midpoint}
For $1\leq i \leq N$,
\begin{align*}
\tilde V_i - v_i^0 & \tleq s(q^0,v^0) \, K_i(q^0) \, \hat  \zeta, \\
\tilde Q_i - q_i^0 & \tleq \max\left(s(q^0,v^0) \, \|v_i^0\|,  s(q^0,v^0)^2 \, K_i(q^0) \right)  \, (\hat \xi-1) . 
\end{align*}\end{theorem}

\begin{remark}
We have numerically estimated the radius of convergence $\hat R$ of the power series $\hat \xi$ and $\hat \zeta$ to obtain  $\hat R \approx 0.094790093$.
\end{remark}

We now consider the discretization of (\ref{eq:Nbodytau}) by a $s$-stage Runge-Kutta scheme with Butcher tableau
\begin{equation}
\label{eq:BT}
\begin{array}{|c}
A \\ \hline
b^T
\end{array}
\end{equation}
where  $A \in \R^{s\times s}$ and $b \in \R^s$. An straightforward generalization of Lemma~\ref{lem:5} and Theorem~\ref{th:hatreg} to Runge-Kutta schemes allows proving the following generalization of Theorem~\ref{th:reg_midpoint}.
\begin{theorem}
Let  
 $$
 (\tilde Q, \tilde V) = (\tilde Q_1,\ldots, \tilde Q_N,\tilde V_1,\ldots, \tilde V_N) \in \R^{6N}[[\Delta\tau]]
 $$
 be the power series expansion of the approximation of the solution $(Q(\Delta\tau),V(\Delta\tau))$ of the initial value problem (\ref{eq:Nbodytau})--(\ref{eq:icondtau}) obtained by applying one step of the Runge-Kutta scheme with Butcher tableau (\ref{eq:BT}). 
\label{th:hatreg2}
For $1\leq i \leq N$,
\begin{align*}
\tilde V_i - v_i^0 & \tleq s(q^0,v^0) \, K_i(q^0) \, \|b\|_{\infty}\, \sum_{k=1}^\infty \hat  \zeta_k\, (2 \, \|A\|_{\infty} \, \Delta\tau)^k, \\
\tilde Q_i - q_i^0 & \tleq \max\left(s(q^0,v^0) \, \|v_i^0\|,  s(q^0,v^0)^2 \, K_i(q^0) \right)  \, \|b\|_{\infty}  \, \sum_{k=1}^\infty \hat  \xi_k\, (2 \, \|A\|_{\infty} \, \Delta\tau)^k.
\end{align*}
\end{theorem}

As a corollary of Theorems~\ref{th:reg2} and \ref{th:hatreg2}, the local error of the application of one step of length $\Delta \tau \in \{ \tau \in \R\  :  \ |\tau| < \hat R/(2\, \|A\|_{\infty})\}$ can be bounded as follows.
\begin{corollary}
If the Runge-Kutta scheme is of order $p\geq 1$, then 
for $1\leq i \leq N$,
\begin{align*}
\|\tilde V_i - V_i\| & \leq s(q^0,v^0) \, K_i(q^0)\, \|b\|_{\infty}\, \sum_{k=p+1}^\infty (\hat  \zeta_k - \zeta_k)\, (2 \, \|A\|_{\infty} \, |\Delta \tau|)^k, \\
\|\tilde Q_i - Q_i\| & \leq \max\left(s(q^0,v^0) \, \|v_i^0\|,  s(q^0,v^0)^2 \, K_i(q^0) \right) \, \|b\|_{\infty}  \, \sum_{k=p+1}^\infty (\hat  \xi_k - \xi_k)\, (2 \, \|A\|_{\infty} \, |\Delta \tau|)^k. 
\end{align*}
\end{corollary}

\begin{remark}
\label{rem:leRK}
This Corollary shows that the local error of a Runge-Kutta method applied with constant step-size $\Delta \tau$ to the time-renormalized $N$-body problem (\ref{eq:Nbodytau})--(\ref{eq:s}) cannot  drastically increase during a close encounter    (provided that $2 \, \|A\|_{\infty} \, |\Delta \tau|$ is small enough). According to the obtained estimates of the local error, only the velocity components of the local error can increase (at a mild rate) as an extreme close encounter occurs, since in that case $s(q^0,v^0) \, K_i(q^0) \sim s(q^0,v^0)^{-1}$. 
\end{remark}

\section{Alternative time-renormalization functions}

Clearly, (\ref{eq:s}) is not the unique globally defined time-reparametrization function that is uniform in the sense that any solution of  (\ref{eq:Nbodytau}) admits an holomorphic extension as a function of the complex time $\tau$ in a strip $\{ \tau \in \C\ : \ |\mathrm{Im}(\tau)| <  R\}$. 
As shown in~\cite{antonana2020}, 
a computationally less complex alternative to function (\ref{eq:s}) is 
\begin{equation}
\label{eq:s1}
s(q,v) = \left(
\sum_{1 \leq i < j \leq N}
\left( \frac{||v_i-v_j||}{||q_i-q_j||}\right)^2+ A(q)\, \sum_{1 \leq i < j \leq N}
 \frac{1}{\|q_i-q_j\|}
\right)^{-1/2}
\end{equation}
where
\begin{equation*}
 A(q) = \sum_{1 \leq i < j \leq N}
 \frac{G\, (m_i + m_j)}{\|q_i-q_j\|^2} .
\end{equation*}
The key observation is that 
\begin{equation*}
  \forall \, 1\leq i < j \leq N, \quad M_{ij}(q) := K_i(q) +K_j(q) \leq A(q),
\end{equation*}
so that all proofs of Sect. ~\ref{sect:Nbptre} and \ref{sect:DottrNbe} remain valid with   (\ref{eq:s1}) instead of (\ref{eq:s}). 

Recall that we chose function (\ref{eq:s}) so that it is, up to a constant factor, a real-analytic lower bound of the estimate of the radius of convergence given in Corollary~\ref{cor:1}.  Actually, in the proofs of Sections~\ref{sect:Nbptre} and \ref{sect:DottrNbe}, this fact is not strictly required. In fact, the essential ingredients of our proofs are the inequalities  (\ref{eq:ss^2}). 
Hence, we may determine $s (q,v)$ as a real-analytic lower bound of 
\begin{equation}
\label{eq:protoS}
\left( \max\left( \max_{1 \leq i < j \leq N}\frac{\|v_i- v_j\|}{ \|q_i - q_j\|}, 
 \max_{1 \leq i < j \leq N}\sqrt{  \frac{M_{ij}(q)}{\alpha\, \|q_i- q_j\|}}\right)\right)^{-1}, 
\end{equation}
with some $\alpha>0$, so that
\begin{equation}
  \label{eq:ss^2bis}
s(q,v) \leq \frac{\|q_i-q_j\|}{\|v_i-v_j\|} \quad \mbox{ and } \quad s(q,v)^2 \leq \frac{\alpha\, \|q_i-q_j\|}{M_{i j}(q)}.
\end{equation}
By replacing the $\infty$-norm of the vector with $(N-1)N$ components in (\ref{eq:protoS}) by its $2p$-norm (for some positive integer $p$) and $M_{ij}(q)$ by $A(q)$, we arrive at
\begin{equation}
\label{eq:s2}
s(q,v) = \left(\sum_{1 \leq i < j \leq N}
         \frac{\|v_i - v_j\|^{2p}}{\|q_i - q_j\|^{2p}}  
         + A(q)^p
         \sum_{1 \leq i < j \leq N} \frac{1}{(\alpha\, \|q_i - q_j\|)^p}
       \right)^{-\frac{1}{2p}}.
\end{equation}

A  time-renormalization that does not depend on the velocities can be derived from (\ref{eq:s2}) by bounding $\|v_i-v_j\|$ in terms of the absolute value  of the potential energy
\begin{equation*}
  U(q) = G\, \sum_{1 \leq i < j \leq N} \frac{m_i\, m_j}{\|q_{i}-q_{j}\|} 
\end{equation*}
and the total energy $E_0=\frac{1}{2} \sum_{i=1}^N m_i\, \|v_i\|^2 - U(q)$. More precisely,
 \begin{equation*}
   \|v_i-v_j\| \leq \sqrt{2} \, (m_{i}^{-1/2}+m_{j}^{-1/2}) (E_0 + U(q)),
 \end{equation*} 
 which leads to
\begin{equation}
  \label{eq:s3}
  \begin{split}
    s(q, E_0) &:=  \left((E_0 + U(q))^p\, \sum_{1 \leq i < j \leq N} \frac{4\, (m_{i}^{-1/2}+m_{j}^{-1/2})^{2p}}{\|q_{i}-q_{j}\|^{2p}}  \right. \\
              &\qquad +  \left. 
                \frac{A(q)^p}{\alpha^p} \, \sum_{1 \leq i < j \leq N} \frac{1}{\|q_{i}-q_{j}\|^p}
              \right)^{-\frac{1}{2p}}.
   \end{split} 
\end{equation}
All the results in Section~\ref{sect:Nbptre} (resp. Section~\ref{sect:DottrNbe}))  can be proven for the two alternative functions (\ref{eq:s2}) and (\ref{eq:s3}),  with different majorant series $\xi, \zeta$ (resp. $\hat \xi, \hat \zeta$) depending on the prescribed parameters $\alpha$ and $p$, having different radius of convergences $R$ (resp. $\hat R$). 
Note however that (\ref{eq:s3}) is no longer valid  in the limit when one of the masses vanishes, and hence it is not expected to perform well with too small mass ratios.

In practice, we suggest to consider $p=2$ or $p=4$, and $\alpha=3$.  That choice for $\alpha$ is motivated by 
comparing the lower and upper bounds of the estimate of the radius of convergence given in Corollary~\ref{cor:1} given in Remark~\ref{rem:fit}), to the following tighter ones
$$ 
\frac{\sqrt{2}-1}{\mu_0 + \sqrt{\nu_0/3}} < \frac{r(\eta_0)}{\sqrt{\mu_0^2 + \nu_0}} <  \frac{0.48}{\mu_0 + \sqrt{\nu_0/3}}.
$$

\section{Numerical experiment}
\label{sect:ne}

In order to illustrate the application of a Runge-Kutta scheme to time-renormalized $N$-body problems, 
we consider the a 15-body model of  the Solar System that includes
\begin{itemize}
\item the Sun, 
\item the Earth-Moon binary considered as mass point centered at its barycenter,
\item the remaining seven planets and Pluto, and
\item  the five main bodies of the asteroid belt: Ceres, Pallas, Vesta, Iris and Bamberga.
\end{itemize}

We consider the initial values at Julian day (TDB) 2440400.5 (the 28th of June of 1969), obtained form the DE430 ephemerides~\cite{Folkner2014}, renormalized so that the center of mass of the 15 bodies is at rest, and run the numerical integrations for $20000$ years. Several close approaches between some of the asteroids occur in that interval of time.

We have applied the 16th order implicit Runge-Kutta method of collocation type with Gauss-Legendre nodes, implemented with fixed point iteration as described in~\cite{antonana2017}. In particular, we have performed our numerical experiments in 
the Julia programming language~\cite{Bezanson2017julia}, using the Julia package IRKGaussLegendre.jl~\cite{IRKGL16} integrated in
SciML/DifferentialEquations.jl~\cite{Rackauckas2017}.

\begin{figure}[ht]
\begin{center}
\resizebox{35em}{!}{\includegraphics{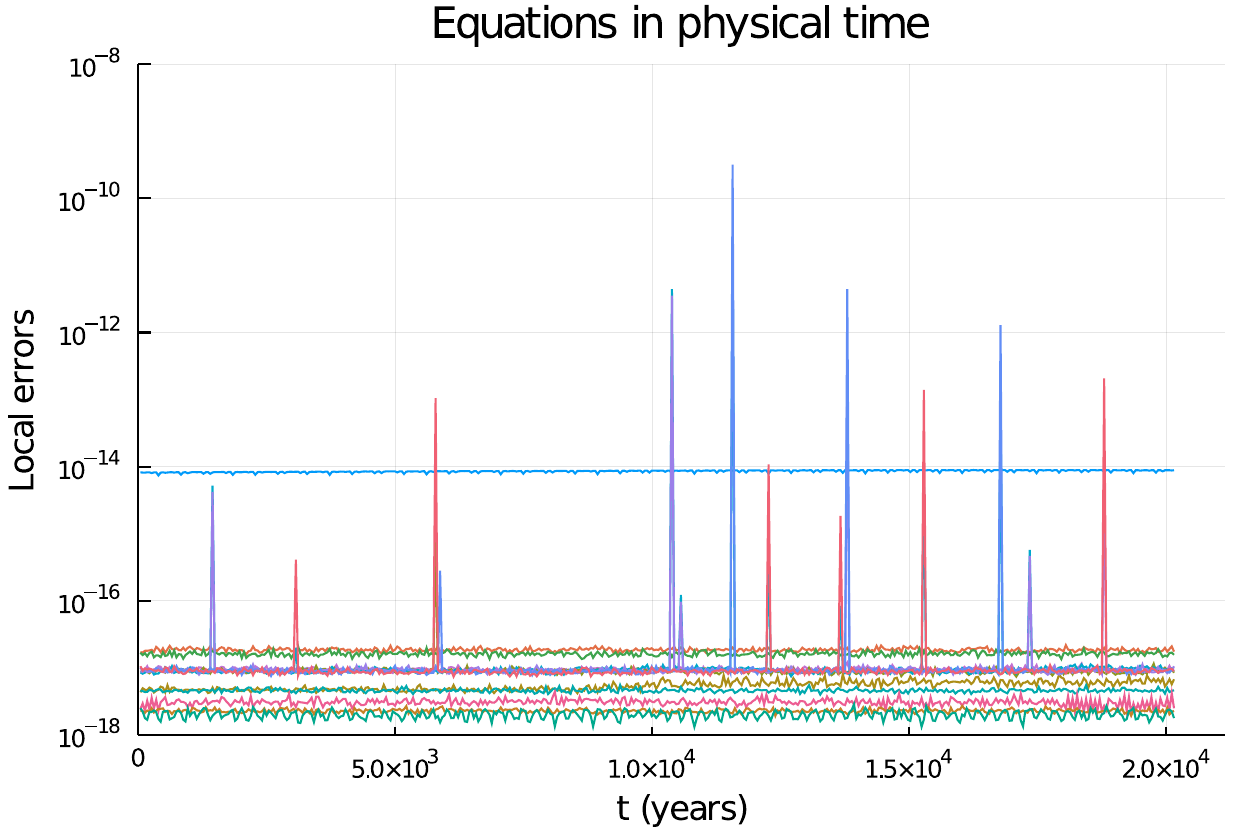}}\\
\resizebox{35em}{!}{\includegraphics{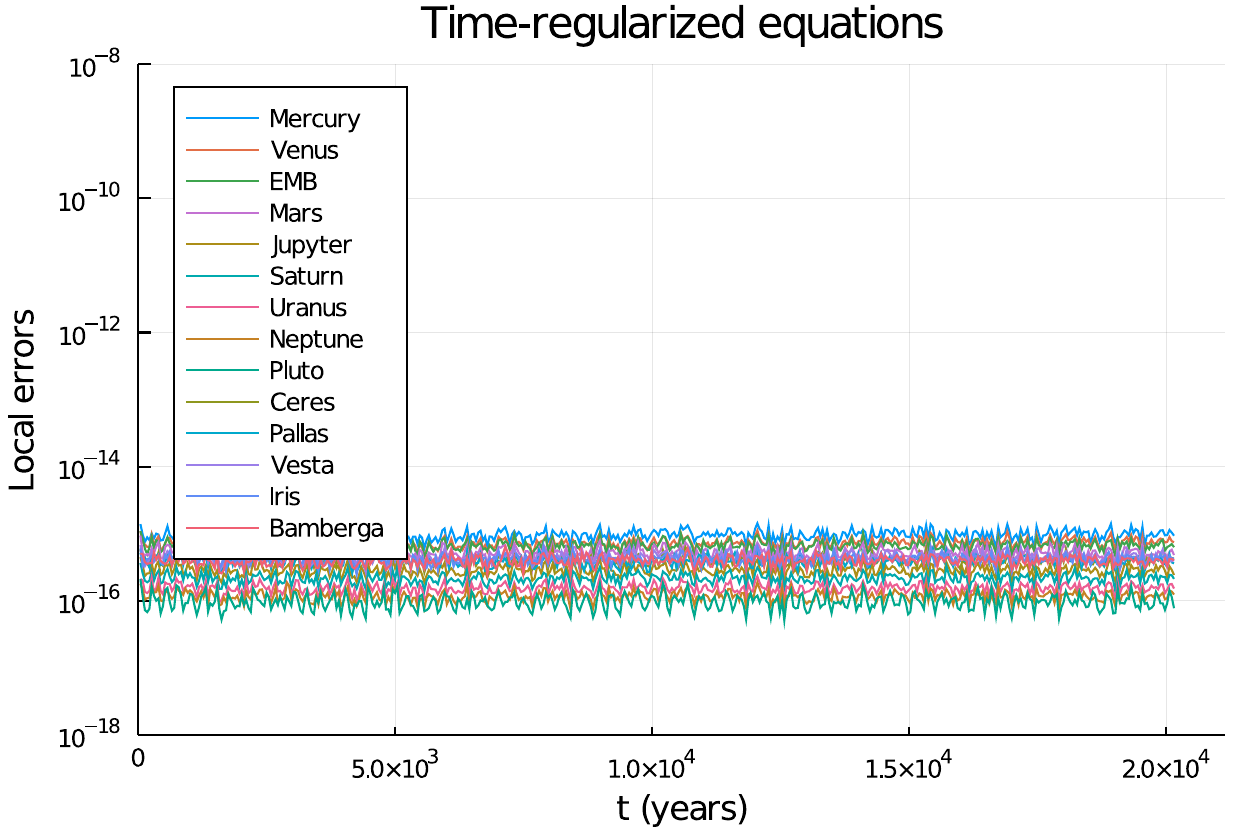}}
\end{center}
\caption{Local errors in position. Top:  integration in physical time with 920000 time-steps of size $\Delta t=8$. Bottom: integration in time-renormalized equations with  920000 time-steps of size $\Delta \tau=0.7196076352409821$.}
\label{fig:local_errors}
\end{figure}

We have first integrated the problem in the equations with physical time (\ref{eq:Nbody}) with a time-step $\Delta t$ of 8 days. The local errors in positions for each of the bodies (except for the sun) are displayed in the upper plot in Figure~\ref{fig:local_errors}. We observe that for most of the steps, the local error is dominated by Mercury's error. Occasionally the errors of two asteroids become considerably larger than Mercury's error, due to a close approach. The highest spike of the local error occur after 10338 years, and is due to a close approach between Pallas and Vesta.
The local errors in velocities (not shown here) present a similar behavior.  
The evolution of global errors in positions are displayed in the upper plot in Figure~\ref{fig:global_errors}.

We then have integrated the time-renormalized problem (\ref{eq:Nbodytau}) with different renormalization functions. 
For a fair comparison with  the integration in physical time, we chose the step-size $\Delta \tau$ in such a way that the same number of steps (and approximately the same CPU time) is required in each case.

We first have tried with the time-renormalization function (\ref{eq:s}), but the spikes of the local errors (not shown here) due to close encounters of the asteroids, although considerably reduced,  do not completely disappear in that case. 
In Figure~\ref{fig:local_errors}, we display the evolution of the local errors in position obtained with the renormalization function (\ref{eq:s2}) with $\alpha=3$ and $p=4$. Observe that there are no spikes of local errors due to close approaches. 
The local errors in velocities (not shown here) neither present such spikes.
In addition, the local errors of Mercury are smaller than those in the integration in physical time. 

However, the local errors of the rest of the bodies (except for pairs of asteroids in a close approach) become larger for the integration in the time-renormalized equations.  In order to understand that, notice that the local errors in physical time of the outer planets are considerably smaller than those of Mercury, because the dominant terms of accelerations of the outer planets are relatively smoother than those of Mercury.
In the time-renormalized equations, the comparatively highly oscillatory motion of Mercury is inherited through the time-renormalized function by the equations of all of the bodies, leading to local errors in positions of similar size for Mercury and the rest of the bodies.

In bottom plot in Figure~\ref{fig:global_errors}, the evolution of global errors for the integration of the time-renormalized equations (for  (\ref{eq:s2}) with $\alpha=3$ and $p=4$) is displayed.  The large errors in the positions of the asteroids due to close approaches observed in  Figure~\ref{fig:global_errors} for the integration in physical time are not present in the integration of the time-renormalized equations.

\begin{figure}[ht]
\begin{center}
\resizebox{35em}{!}{\includegraphics{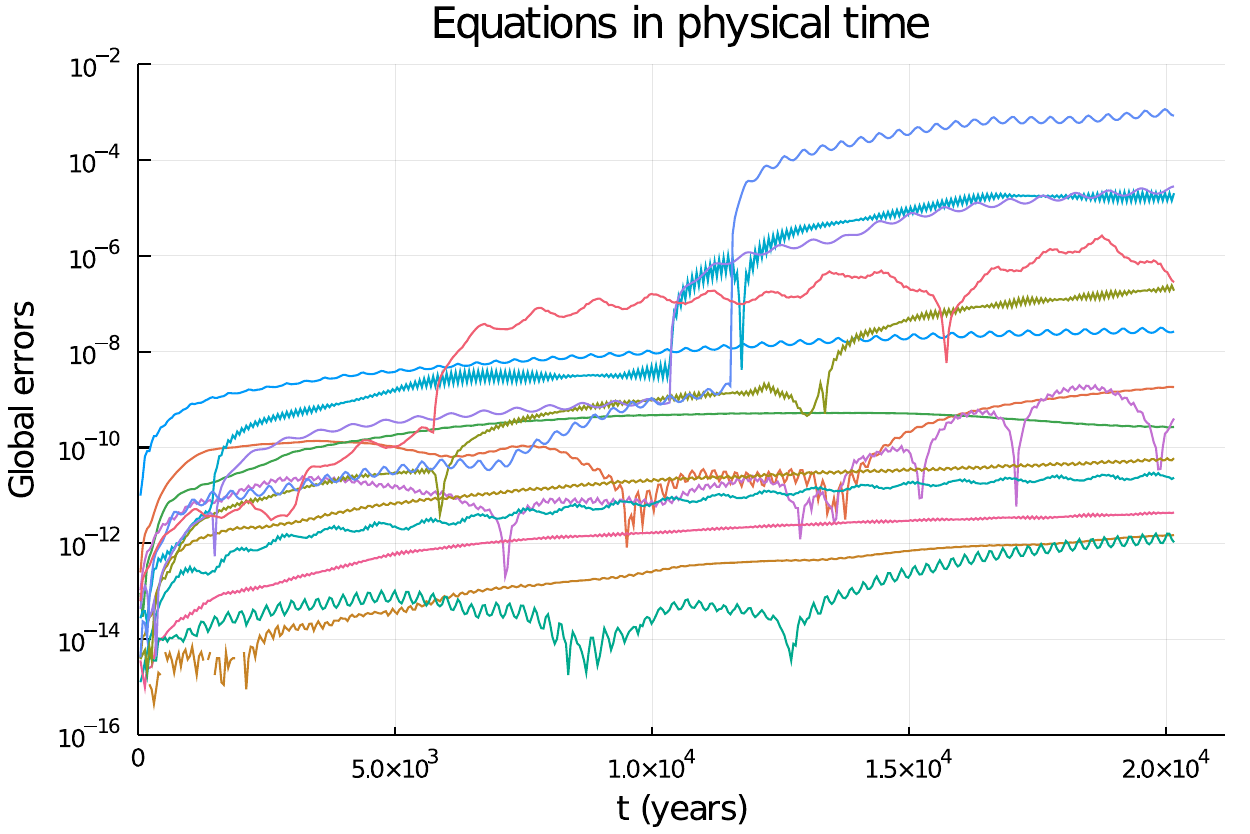}}\\
\resizebox{35em}{!}{\includegraphics{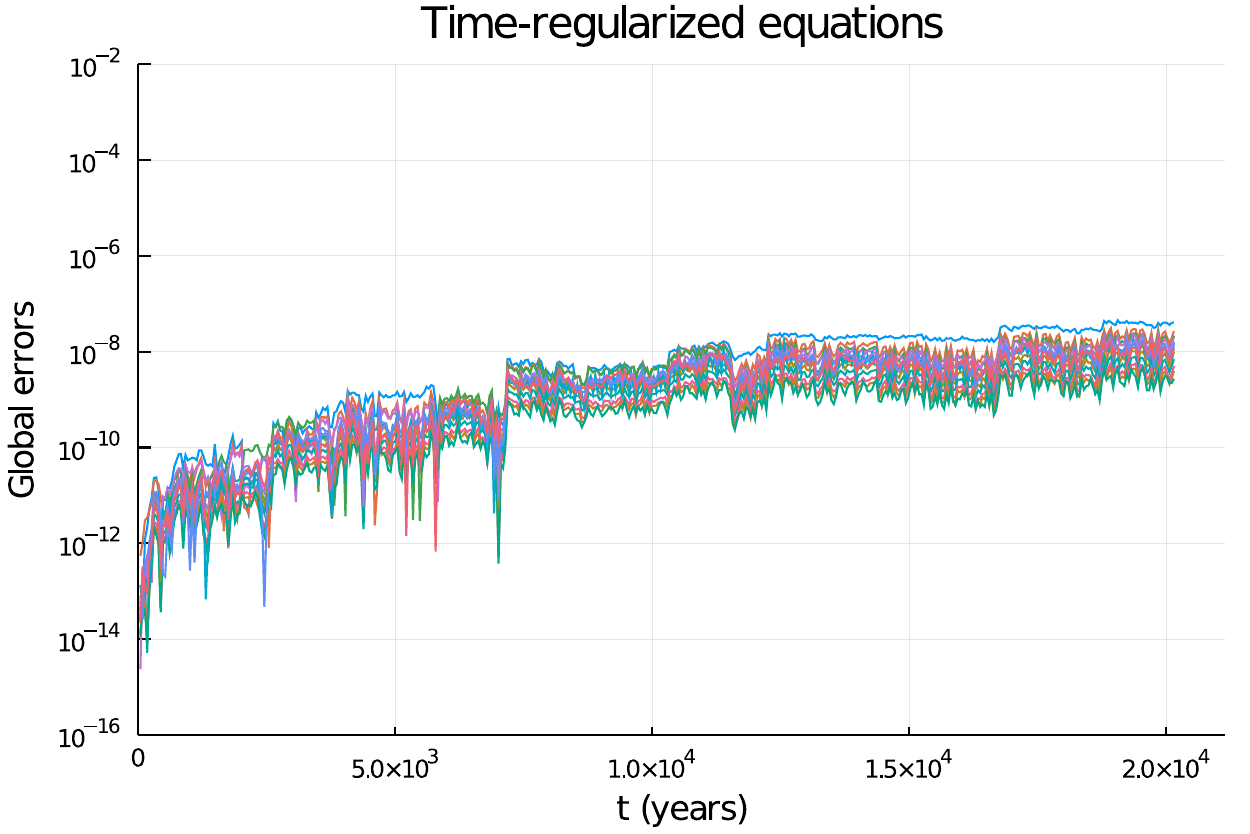}}
\end{center}
\caption{Global errors in position.  Top: integration in physical time with 920000 time-steps of size $\Delta t=8$. Bottom: integration in time-renormalized equations with  920000 time-steps of size $\Delta \tau=0.7196076352409821$}
\label{fig:global_errors}
\end{figure}

\section*{Acknowledgements}

Ma and AM have received funding by the Spanish State Research Agency through project PID2019-104927GB-C22 (AEI/FEDER, UE) with acronym GNI-QUAMC, and also from the Department of Education of the Basque Government through the Consolidated Research Group MATHMODE (IT1294-19).

\bibliographystyle{alpha}

\begin{thebibliography}{10}


\bibitem{antonana2017} {\sc M.~ Anto\~nana,  J.~ Makazaga, A.~Murua}, {\em Reducing and monitoring round-off error propagation for symplectic implicit Runge-Kutta schemes}, Numerical Algorithms, 76:  861--880, 2017. 



\bibitem{antonana2020} {\sc M.~ Anto\~nana,  P.~Chartier,  J.~Makazaga, A.~ Murua}, {\em Global time-regularisation of the gravitational N-body problem}, SIAM J. Appl. Dyn. Syst., 19(4):  2658-2681, 2020. 



\bibitem{IRKGL16}
Anto\~nana, M.: {IRKGaussLegendre.jl: Implicit Runge-Kutta Gauss-Legendre 16th
 order method}.
\newblock 
\url{https://github.com/SciML/IRKGaussLegendre.jl}

\bibitem{Bezanson2017julia}
Bezanson, J., Edelman, A., Karpinski, S., Shah, V.B.: Julia: A fresh approach
 to numerical computing.
\newblock SIAM review \textbf{59}(1), 65--98 (2017).
\newblock 
\url{https://doi.org/10.1137/141000671}


\bibitem{Car61} {\sc H.~Cartan}, {\em Th\'eorie \'el\'ementaire des fonctions analytiques d'une et plusieurs variables}, Hermann, 1961.

\bibitem{Folkner2014} {\sc W.M.~Folkner,  J.G.~Williams, D.H.~Boggs, R.S.~Park,  P.~Kuchynka}, {\em The planetary and
lunar ephemerides de430 and de431}, IPN Progress Report, 42, 2014.

\bibitem{JvdH} {\sc J. van der Hoeven}, {\em Majorants for formal power series}, unpublished notes. 

\bibitem{vK75} {\sc S.~von Kowalevsky}, {\em Zur theorie der partiellen differentialgleichungen}, J. Reine und Angew. Math., 80:1--32, 1875.





\bibitem{Pet50} {\sc I.G.~Petrovsky}, {\em Lectures on Partial Differential Equations}, Interscience Publishers, 1950.


\bibitem{Rackauckas2017}
Rackauckas, C., Nie, Q.: Differentialequations.jl--a performant and
 feature-rich ecosystem for solving differential equations in julia.
\newblock Journal of Open Research Software \textbf{5}(1) (2017)



\bibitem{sundman} {\sc K. F.~Sundman}, {\em M\'emoire sur le probl\`eme des trois corps},  Acta Mathematica. 36:  105--179, 1912.


\bibitem{wang} {\sc Q.-D. Wang}, {\em The global solution of the $n$-body problem}, Celestial Mechanics and Dynamical Astronomy, 50 (1): 73--88, 1991.

\bibitem{hairer06} {\sc E.~Hairer, C.~Lubich, and G.~Wanner}, {\em Geometric {N}umerical {I}ntegration. {S}tructure-{P}reserving {A}lgorithms for {O}rdinary {D}ifferential {E}quations}, Springer-Verlag, {S}econd~ed., 2006.
\bibitem{leimkuhler05}
{\sc B.  Leimkuhler, S. Reich}, {\em Simulating Hamiltonian Dynamics},  Cambridge Monographs on Applied and Computational Mathematics, Cambridge University Press, 2005. 
\bibitem{hairer97}
{\sc E. Hairer}, {\em Variable time step integration with symplectic methods}, 
Applied Numerical Mathematics. 1997, vol. 25, no. 2-3, p. 219-227 
\bibitem{huang97}
{\sc W. Huang, B. Leimkuhler}, {\em The Adaptive Verlet Method}, SIAM J. Sci. Comp., vol. 18, 1997. 
\bibitem{mikkola97}
{\sc S. Mikkola}, {\em Practical Symplectic Methods with Time Transformation for the Few-Body Problem}, CelMechDynAstr, vol. 67, 1997.
\bibitem{calvo98}
{\sc M. P. Calvo, M.A. L\' opez-Marcos and J.M. Sanz-Serna}, {\em Variable step implementation of geometric integrators}, vol. 28, 1998, Applied Numerical Mathematics. 
\bibitem{hairer05}
{\sc E. Hairer, G. S\" oderlind}, {\em Explicit, time-reversible, adaptive step-size control}, SIAM J. Sci. Comp. vol. 26, 2005. 

\end{thebibliography}

\end{document}